\documentclass{amsart}
\date{December 6, 2014}
\usepackage{graphicx} 
\usepackage{verbatim,enumerate}
\title{%
    Intrinsic properties of surfaces with singularities
}
\author{M.~Hasegawa}
\author{A.~Honda}
\author{K.~Naokawa}
\author{K.~Saji}
\author{M.~Umehara}
\author{K.~Yamada}
\address[Masaru Hasegawa]{%
  Instituto de Ci\^encias Matem\'aticas e 
  de Computa\c{c}\~ao - USP,
  Avenida Trabalhador Sao-carlense, 400- Centro,
  CEP:13566-590 - S\~ao Carlos - SP, Brazil.
}
\email{mhasegawa@icmc.usp.br}

\dedicatory{To the memory of Professor Shoshichi Kobayashi}
\address[Atsufumi Honda]{
  Miyakonojo National College of Technology, 
  473-1, Yoshiocho, Miyakonojo, 
  Miyazaki 885-8567,
  Japan}
\email{atsufumi@cc.miyakonojo-nct.ac.jp}

\address[Kosuke Naokawa and Kentaro Saji]{%
  Department of Mathematics,
  Faculty of Science,
  Kobe University,
  Rokko, Kobe 657-8501, Japan}
\email{naokawa@port.kobe-u.ac.jp}
\email{saji@math.kobe-u.ac.jp}

\address[Masaaki Umehara]{
  Department of Mathematical and Computing Sciences,
  Tokyo Institute of Technology,
  2-12-1-W8-34, O-okayama Meguro-ku,
  Tokyo 152-8552, Japan}
\email{umehara@is.titech.ac.jp}
\address[Kotaro Yamada]{%
  Department of Mathematics,
  Tokyo Institute of Technology,
  O-okayama, Meguro, Tokyo 152-8551, Japan%
}
\email{kotaro@math.titech.ac.jp}

\thanks{%
  The first author was supported by the 
  FAPESP post-doctoral grant number 2013/02543-1. 
  The third author was partly supported by 
  the Grant-in-Aid for JSPS Fellows.
  The fourth, fifth and sixth authors  were 
  partially supported by Grant-in-Aid for 
  Scientific Research (C) No.~26400087, 
  Scientific Research (A) No.~262457005, 
  and Scientific Research (C) No.~26400006,
  respectively, 
  from the Japan Society for the Promotion of Science.}
  
\numberwithin{equation}{section}
\newtheorem{Thm}{Theorem}[section]
\newtheorem{Fact}[Thm]{Fact}

\newtheorem{Cor}[Thm]{Corollary}
\newtheorem{Lemma}[Thm]{Lemma}
\newtheorem{Prop}[Thm]{Proposition}
\theoremstyle{definition}
\newtheorem{Def}[Thm]{Definition}
\newtheorem{Exa}[Thm]{Example}
\theoremstyle{remark}
\newtheorem{Rmk}[Thm]{Remark}

\newcommand{\vect}[1]{\boldsymbol{#1}}
\newcommand{\op}[1]{{\operatorname{#1}}}
\newcommand{\dy}{\displaystyle}

\newcommand{\inner}[2]{\left\langle{#1},{#2}\right\rangle}
\newcommand{\R}{{\vect{R}}}
\newcommand{\E}{\mathcal{E}}
\newcommand{\T}{\mathcal{T}}
\newcommand{\J}{\mathcal{J}}
\newcommand{\M}{\mathcal{M}}
\newcommand{\N}{\mathcal{N}}
\newcommand{\X}{\mathfrak{X}}
\renewcommand{\O}{\operatorname{O}}
\newcommand{\id}{\operatorname{id}}
\newcommand{\Hess}{\operatorname{Hess}}
\newcommand{\sgn}{\operatorname{sgn}}

\newcommand{\pmt}[1]{{\begin{pmatrix} #1  \end{pmatrix}}}

\renewcommand{\phi}{\varphi}
\renewcommand{\epsilon}{\varepsilon}
\begin{document}
\maketitle
\begin{abstract}
 In this paper, we give two classes of positive semi-definite
 metrics on 2-manifolds.
 The one is called a class of {\it Kossowski metrics\/} and the
 other is called a class of {\it Whitney metrics}:
 The pull-back metrics of wave fronts which admit
 only cuspidal edges and swallowtails in $\R^3$
 are Kossowski metrics, and
 the pull-back metrics of surfaces consisting
 only of cross cap singularities are
 Whitney  metrics.
 Since the singular sets of Kossowski metrics are 
 the union of regular curves
 on the domains of definitions,
 and Whitney metrics admit only isolated singularities,
 these two classes of metrics are disjoint.
 In this paper, we give several characterizations of
 intrinsic invariants of cuspidal edges and cross caps
 in these classes of metrics.
 Moreover, we prove Gauss-Bonnet type formulas
 for  Kossowski metrics and for
 Whitney metrics on compact 
 $2$-manifolds.
\end{abstract}
\section{Introduction}
\label{intro}
Let $U$ be a domain in the $uv$-plane
and $f:U\to \R^3$ a $C^\infty$-map.
A point $p\in U$ is called a {\it singular point\/}
if $f$ is not an immersion at $p$.
If $f$ does not admit any singular points
(i.e.\ $f$ is an immersion), $f$ is called 
a {\it regular surface}.
We fix such a regular surface $f$ and 
denote by
\begin{equation}\label{eq:first}
   ds^2=E\,du^2+2F\, du\,dv+G\, dv^2
\end{equation}
the {\em first fundamental form\/} 
(or the {\em induced metric}) of $f$, where
\begin{equation}
\label{eq:first2}
   E:=f_u\cdot f_u,\qquad F:=f_u\cdot f_v,\qquad
   G:=f_v\cdot f_v.
\end{equation}
Here, ``$\cdot$'' denotes the canonical inner product of $\R^3$.
We let $\nu$ be the unit normal vector field of $f$ and set
\[
    L:=f_{uu}\cdot \nu,\qquad M:=f_{uv}\cdot \nu,\qquad
    N:=f_{vv}\cdot \nu.
\]
Then the {\em second fundamental form} of $f$ is 
given by
\[
    h=L\,du^2+2M\, du\,dv+N\, dv^2.
\]
An invariant of regular surfaces
is called {\it intrinsic\/} if it can be
reformulated as an invariant of the first fundamental forms.
For example, the Gaussian curvature
$K:=(LN-M^2)/(EG-F^2)$ is an intrinsic
invariant, in fact, it coincides with
the sectional curvature of the metric $ds^2$
and has the expression 
\begin{multline}\label{eq:egregium}
  K=\frac{E\bigl(E_vG_v-2F_uG_v+(G_u)^2\bigr)}{4(EG-F^2)^2} \\
      +\frac{F(E_uG_v-E_vG_u-2E_vF_v-2F_uG_u+4F_uF_v)}{4(EG-F^2)^2}\\
      +\frac{G\bigl(E_uG_u-2E_uF_v+(E_v)^2\bigr)}{4(EG-F^2)^2}
             -\frac{E_{vv}-2F_{uv}+G_{uu}}{2(EG-F^2)}.
\end{multline}
On the other hand, 
an invariant of a $C^\infty$-map $f$ is 
called {\it extrinsic\/} if there exists another 
$C^\infty$-map $g$ whose induced metric coincides with
that of $f$ but the corresponding invariant of $g$
takes different values.
For example, the mean curvature $H$ of regular surfaces
is an extrinsic invariant, since we know that
a plane admits an isometric deformation varying $H$. 
By definition, extrinsic invariants cannot be
intrinsic invariants.
Moreover, the converse statement
is true for the real-analytic case: 

\begin{Prop}
 In the class of real analytic regular surfaces, 
 each non-extrinsic invariant is intrinsic. 
\end{Prop}

\begin{proof}
 Let $C^\omega_0(\R^2,\R^3)$
 be the set of germs of real analytic immersions
 of $(\R^2,o)$ into $(\R^3,o)$, 
where $o$ is the origin.
 A map 
 $I:C^\omega_0(\R^2,\R^3)\to \R$
 is called an {\it invariant\/} if it does not depend on the 
 choice of a local coordinate
 system of regular surfaces
 and does not change values for
 any Euclidean motions in $\R^3$.
 We denote by 
 $\M^2_+$ 
 the set of germs of positive definite real analytic metrics
 defined at the origin $o\in \R^2$.
 For each $d\sigma^2\in \M^2_+$,
 the classical Janet-Cartan theorem
 (cf.\ \cite{Sp}) implies that 
 there exists a neighborhood $U$
 of the origin $o$ and a real analytic
 immersion
 $f:U\to \R^3$
 such that the pull-back of the
 canonical metric of $\R^3$ coincides
 with $d\sigma^2$. 
 Suppose that $I$ is not an extrinsic invariant.
 We denote by $I(f)$ the invariant of
 $f$ at $o$. Since $I$ is not extrinsic,
 the value $I(f)$ does not depend
 on the choice of such $f$. 
 So the map
 $\M^2_+\ni d\sigma^2 \mapsto I(f)\in \R$
 is well-defined, which means that
 $I(f)$ is an intrinsic invariant.
\end{proof}

In \cite{HHNUY} 
and \cite{MSUY}, several geometric invariants
of cross caps (see Section \ref{sec:W} 
for definition) 
and wave front singularities\footnote{The
definition of wave fronts is given in the appendix.}
were introduced,
and it was shown that some of them are 
actually {\it intrinsic\/} invariants, 
by
\begin{enumerate}
 \item[(i)] setting up a class of local coordinate
	    systems determined by 
	    the induced metrics (i.e.\ the first fundamental forms),
\item[(ii)] and giving formulas for the invariants
	    in terms of the coefficients of the first
	    fundamental forms with
	    respect to the above coordinate systems.
\end{enumerate}
Moreover, like as in the case of 
regular surfaces, one can expect the 
existence of a suitable class of positive 
semi-definite metrics for a given class of
singularities.
Fortunately, Kossowski \cite{K} defined a
class of positive semi-definite metrics
which characterizes the non-degenerate front singularities
(the non-degeneracy of singularities of fronts
is defined in the appendix).
In Section~\ref{sec:K}, we call metrics in such a class  
{\it Kossowski metrics}, and will describe the intrinsic 
invariants of cuspidal edges
(the definition of cuspidal edges is given in the
appendix) 
shown in \cite{MSUY} 
in this class of metrics (cf. Remark \ref{rmk:expression}).
Moreover, we show Gauss-Bonnet type formulas
for this class of metrics in Section~\ref{sec:GB}.

In Section~\ref{sec:W}, we introduce another new class of
positive semi-definite metrics
called {\it Whitney metrics},
which characterizes the cross cap singularities
in $\R^3$, and reformulate intrinsic invariants
of cross caps given in \cite{HHNUY} in terms of 
Whitney metrics.
Moreover, in Section~\ref{sec:WGB},
we prove a Gauss-Bonnet type formula
for this class of metrics.

\section{Kossowski metrics and cuspidal edges}
\label{sec:K}

In the  first part of this section, we 
introduce a class of positive semi-definite metrics 
describing the properties of wave fronts 
(see the appendix) intrinsically.
This class of metrics was defined by Kossowski \cite{K}.
For this purpose, we fix  a $2$-manifold $M^2$,
and a positive semi-definite metric $d\sigma^2$
on $M^2$.  A point $p\in M^2$ is called a 
{\it singular point\/} of 
the metric $d\sigma^2$ if the metric is not positive definite at $p$.
We denote by $\X$ the set of smooth vector fields
on $M^2$, and by $C^{\infty}(M^2)$ the set of $\R$-valued smooth functions on $M^2$.

We set
$\inner{X}{Y}:=d\sigma^2(X,Y)$.
Kossowski \cite{K} defined a map
$\Gamma:\X\times \X\times \X\to C^{\infty}(M^2)$
as 
\begin{multline}\label{eq:Gamma}
\Gamma(X,Y,Z):=\frac12
 \biggl(
 X\inner{Y}{Z}+Y\inner{X}{Z}-Z\inner{X}{Y}\\
 +\inner{[X,Y]}{Z}-\inner{[X,Z]}{Y}
 -\inner{[Y,Z]}{X}
 \biggr)
\end{multline}
and showed that it plays an important role
in giving an intrinsic characterization of 
generic wave fronts.
So, we call $\Gamma$, a
{\it Kossowski pseudo-connection}.

If the metric $d\sigma^2$ is positive definite, then  
\begin{equation}\label{eq:GN}
\Gamma(X,Y,Z)=\langle \nabla_XY,Z\rangle
\end{equation}
holds, where
$\nabla$ is the Levi-Civita connection of $d\sigma^2$.
One can easily check the following two identities (cf. \cite{K})
\begin{align}\label{eq:1}
&X\langle Y,Z\rangle=\Gamma(X,Y,Z)+\Gamma(X,Z,Y),\\
\label{eq:2}
&\Gamma(X,Y,Z)-\Gamma(Y,X,Z)=\langle [X,Y],Z\rangle.
\end{align}
The equation \eqref{eq:1} corresponds to the condition 
that  $\nabla$ is a metric connection, and
the equation \eqref{eq:2}
corresponds to the condition that
$\nabla$ is torsion free.
The following assertion can be also
easily verified: 

\begin{Prop}[Kossowski \cite{K}]\label{prop:K-conn}
 For each $Y\in \X$
 and for each $p\in M^2$,
 the map
 \[
     T_pM^2\times T_pM^2\ni (v_1,v_2)\longmapsto
           \Gamma(V_1,Y,V_2)(p)\in \R
 \]
is a well-defined bi-linear map, 
where $V_j$ $(j=1,2)$
are vector fields of $M^2$ 
satisfying $v_j=V_j(p)$.
\end{Prop}

For each $p\in M^2$, the subspace
\begin{equation}\label{eq:N}
 \N_p:=\biggl\{v\in T_pM^2\,;\, d\sigma^2(v,w)=0
\mbox{ for all $w\in T_pM^2$}
\biggr\}
\end{equation}
is called the {\it null space} or the {\it radical} at $p$.
A non-zero vector which belongs to $\N_p$ is called
a {\it null vector\/} at $p$.

\begin{Lemma}[Kossowski \cite{K}]\label{lem:K-conn}
 Let $p$ be a singular point of $d\sigma^2$.
 Then the Kossowski pseudo-connection $\Gamma$
 induces a tri-linear map
 \[
    \hat \Gamma_p:T_pM^2\times T_pM^2\times \N_p
            \ni (v_1,v_2,v_3) \longmapsto
          \Gamma(V_1,V_2,V_3)(p)\in \R,
 \]
where
$V_j$ $(j=1,2,3)$ are vector fields of $M^2$
such that $v_j=V_j(p)$.

\end{Lemma}
\begin{proof}
 Applying \eqref{eq:Gamma}, 
 \begin{align*}
  & 2\Gamma(V_1,fV_2,V_3) \\
  &\phantom{aaa}=
  V_1\langle fV_2,V_3\rangle+
fV_2\langle V_1,V_3\rangle-V_3\langle V_1,fV_2
\rangle\\
  &\phantom{aaaaaaaaa}+
\langle [V_1,fV_2],V_3\rangle-\langle [V_1,V_3],fV_2
\rangle
  -\langle [fV_2,V_3],V_1\rangle \\
  &\phantom{aaa}=
  2f\Gamma(V_1,V_2,V_3)
  +(V_1f) \langle V_2,V_3\rangle-
  (V_3f)\langle V_1,V_2\rangle \\
  &\phantom{aaaaaaaaaaaaaaaaaaaaaaaaa}+
(V_1f)
\langle V_2,V_3\rangle +(V_3f) \langle V_2,V_1\rangle\\
  &\phantom{aaa}=
  2f\Gamma(V_1,V_2,V_3)
    +2(V_1f) \langle V_2,V_3\rangle= 2f\Gamma(V_1,V_2,V_3)
 \end{align*}
holds at $p$,
 where the fact that $V_3(p)\in \N_p$ 
is used to show the last equality.
\end{proof}
\begin{Def}\label{def:adms}
 A singular point $p$ of the metric $d\sigma^2$
 is called 
 {\it admissible\footnote{
 The admissibility was originally introduced by 
 Kossowski \cite{K}. He called it
 $d(\langle,\rangle)$-{\it flatness}.
}%
\/} 
 if $\hat \Gamma_p$ in Lemma~\ref{lem:K-conn} vanishes.
 If each singular point of $d\sigma^2$
 is admissible, then $d\sigma^2$ is called
 an {\it admissible metric}.
\end{Def}

We are interested in admissible metrics
because of the following fact.
\begin{Prop}\label{prop:exa-adms}
 Let $f:M^2\to \R^3$ be a $C^\infty$-map.
 Then the induced metric $d\sigma^2(:=df\cdot df)$ 
by $f$ on $M^2$
 is an admissible metric.
\end{Prop}
\begin{proof}
 Let $D$ be the Levi-Civita connection of the
 canonical metric of $\R^3$.
Then the Kossowski pseudo-connection of
$d\sigma^2$ is given by  (cf. \eqref{eq:GN})
 \[
    \Gamma(X,Y,Z)=D_X df(Y)\cdot df(Z),
 \]
which vanishes if $df(Z_p)=0$, proving the
 assertion.
\end{proof}

A singular point of the metric $d\sigma^2$
is called of {\it rank one} if $\N_p$
is a $1$-dimensional subspace of $T_pM^2$.
\begin{Def}\label{def:adms-coord}
 Let $d\sigma^2$ be a positive semi-definite metric on $M^2$.
 A local coordinate system
 $(U; u,v)$ of $M^2$
 is called {\it adjusted\/} at a singular point $p\in U$ 
 if  
 \[
     \partial_v:=\partial/ \partial v
 \]
 belongs to $\N_p$.
 Moreover, if $(U; u,v)$ is adjusted at each singular point
 of $U$, it is called an {\it adapted local coordinate system\/} of 
 $M^2$.
\end{Def}

By a suitable affine transformation
in the $uv$-plane,
one can take a local coordinate system
which is adjusted at a given rank one singular point $p$.

\begin{Lemma}\label{lem:u2}
 Let $(\xi,\eta)$ and $(u,v)$ be two local coordinate 
 systems centered at a rank one singular point $p$.
 Suppose that $(u,v)$ is adjusted at $p=(0,0)$.  
 Then $(\xi,\eta)$ 
 is also adjusted at $p$
 if and only if 
 \begin{equation}\label{eq:u2}
  u_\eta(0,0)=0
 \end{equation}
 holds. 
\end{Lemma}
\begin{proof}
 It holds that
 $\partial_\eta=u_\eta \partial_u+v_\eta \partial_v$.
 If $\partial_v\in \N_p$,
 then  $\partial_\eta\in \N_p$
 if and only if $u_\eta$ vanishes at $p$.
\end{proof}

The following assertion gives a characterization
of admissible singularities:
\begin{Prop}\label{prop:property}
 Let $(u,v)$ be a local coordinate system
 centered at a rank one singular point $p$.
 If $p$ is admissible and $(u,v)$ is adjusted at $p$, 
 then 
 \begin{equation}\label{eq:admissible}
   F=G=0,\quad E_v = 2 F_u, \quad G_u=G_v=0 
 \end{equation}
 hold at $p=(0,0)$, where $d\sigma^2=E\,du^2+2F\,du\,dv+G\,dv^2$. 
 Conversely, if there exists a local coordinate system $(u,v)$
 centered at $p$ satisfying \eqref{eq:admissible}, then
 $p$ is an admissible singular point, and
 $(u,v)$ is adjusted at $p$.
\end{Prop}
\begin{proof}
 Since $[\partial_u,\partial_v]$ vanishes,
 and $\partial_v\in \N_p$ at $p$,
 the formula \eqref{eq:Gamma} yields that
 \begin{align*}
  &2\hat\Gamma(\partial_u,\partial_u,\partial_v)
  =2\partial_u\langle \partial_u,\partial_v \rangle-
  \partial_v\langle \partial_u,\partial_u \rangle=2F_u-E_v,\\
  &2\hat\Gamma(\partial_u,\partial_v,\partial_v)
  =\partial_u\langle \partial_v,\partial_v \rangle+
  \partial_v\langle \partial_u,\partial_v \rangle
  - \partial_v\langle \partial_u,\partial_v \rangle=
  \partial_u\langle \partial_v,\partial_v \rangle
  =G_u,\\
  &2\hat\Gamma(\partial_v,\partial_v,\partial_v)
  =\partial_v\langle \partial_v,\partial_v \rangle=G_v
 \end{align*}
 hold at the origin $(0,0)$.
 Thus $\hat \Gamma_p$ vanishes if
 and only if
 \eqref{eq:admissible} holds at $p$.
\end{proof}

In this section, we are interested in the case that
the set of singular points (called the {\em singular set})
consists of a regular curve on the domain.
The following assertion plays an important role
in the latter discussions.

\begin{Cor}\label{cor:property}
 Let $(u,v)$ be a local coordinate system of
 $M^2$ such that the $u$-axis is a singular set
 and $\partial_v$ is a null vector along the $u$-axis.
 Then all points of the $u$-axis are admissible
 singular points if and only if
 \begin{equation}\label{eq:admissible2}
  E_v=G_v=0
 \end{equation}
 holds on the $u$-axis.
\end{Cor}
\begin{proof}
 Since  $\partial_v$ is a null-vector field along
 the $u$-axis, we have that
 $F(u,0)=G(u,0)=0$.
 Differentiating it with respect to $u$, we get
 $F_u(u,0)=G_u(u,0)=0$.
 Then the assertion follows from \eqref{eq:admissible}.
\end{proof}

\begin{Def}\label{def:frontal}
 An  admissible metric $d\sigma^2$ defined on $M^2$
 is called a {\it frontal metric\/} if for each local coordinate
 system $(U;u,v)$, there exists a $C^\infty$-function
 $\lambda(u,v)$ on $U$ such that
 \begin{equation}\label{eq:lambda}
  EG-F^2=\lambda^2,
 \end{equation}
 where
 $d\sigma^2=E\,du^2+2F\,du\,dv+G\,dv^2$
 is a local expression of the metric $d\sigma^2$ on $U$.
\end{Def}

The following assertion is the reason of
the naming of frontal metrics.

\begin{Prop}\label{prop:frontal}
Let $f:M^2\to \R^3$ be a frontal
$($see the appendix for the definition of frontals$)$.
 Then the induced metric $d\sigma^2(:=df\cdot df)$ 
on $M^2$  is a frontal  metric.
\end{Prop}

\begin{proof}
Let $(U;u,v)$ be a sufficiently small
local coordinate neighborhood of $M^2$. 
We can take
a smooth unit normal vector field
$\nu$ of $f$ on $U$.
Let $E,F,G$ be
the coefficients of the first fundamental form 
as in \eqref{eq:first2}.
Then it holds that
$$
EG-F^2=\lambda^2,
\qquad \lambda:=\det(f_u,f_v,\nu),
$$
which proves the assertion.
\end{proof}

From now on, we fix a frontal metric $d\sigma^2$ on $M^2$. 
We set
\begin{equation}\label{eq:dA}
 d\hat A:=\lambda\, du\wedge dv,
\end{equation}
where $\lambda$ is the function given in \eqref{eq:lambda}.
If one can choose the function
$\lambda$ for each local coordinate system $(U;u,v)$ 
so that $d\hat A$ is a
smooth $2$-form on $M^2$, 
the frontal metric $d\sigma^2$ is called 
{\it co-orientable}.
In this case, we call 
$d\hat A$ the {\it signed area element\/} 
associated to $d\sigma^2$.

On the other hand, suppose $M^2$ is
oriented, and $(u,v)$ is a local
coordinate system which is
compatible with respect to the 
orientation.
Then the form
\begin{equation}\label{eq:dA0}
    dA:=\sqrt{EG-F^2}\,du\wedge dv=|\lambda|\,du\wedge dv
\end{equation}
does not depend on the choice of $(u,v)$
and gives a 
continuous $2$-form on $M^2$.
The existence of $dA$ is equivalent to the orientability of $M^2$.
We call $dA$ the (un-signed) {\it area element\/} 
associated to $d\sigma^2$.
The area element $dA$ vanishes at 
the singular set $\Sigma$ of
$d\sigma^2$, and is not differentiable on $\Sigma$ in general.

If $M^2$ is simply connected, all frontal metrics
on $M^2$ are orientable and co-orientable.
The co-orientability of frontal metrics
is related to that of
frontals in $\R^3$ as follows
(the definition of frontals are given in the appendix):

\begin{Prop}\label{prop:co-o}
Let $f:M^2\to \R^3$ be a frontal.
Then the induced metric $d\sigma^2(:=df\cdot df)$
is co-orientable if and only if
so is $f$ $($the co-orientability of $f$
is defined in the appendix$)$.
\end{Prop}

\begin{proof}
Suppose that $f$ is co-orientable.
Then we can take a 
unit normal vector field $\nu$ of $f$
defined on $M^2$.
Let $(U;u,v)$ be a local coordinate system.
By setting
$\lambda:=\det(f_u,f_v,\nu)$, \eqref{eq:lambda} holds. 
Moreover, the $2$-form given in
\eqref{eq:dA} is defined on $M^2$.
So $d\sigma^2$ is co-orientable.

We next assume that $d\sigma^2$ is
co-orientable. 
We can take 
an atlas $\{(U_\alpha;u_\alpha,v_\alpha)\}_{\alpha\in \Lambda}$ of the manifold $M^2$
so that there exists a function 
$\lambda_\alpha$ defined on $U_\alpha$
satisfying \eqref{eq:lambda}.
Here, there is a $\pm$-ambiguity of
the sign of the function $\lambda_\alpha$
on each local coordinate $(U_\alpha;u_\alpha,v_\alpha)$.
Since $d\sigma^2$ is co-orientable,
we can fix a signed area
element $d\hat A$ defined on $M^2$, and
each $\lambda_\alpha$ can be uniquely chosen 
so that 
\begin{equation}\label{eq:star}
\lambda_\alpha du_\alpha\wedge dv_\alpha=d\hat A
\end{equation}
holds. 
We may assume that each $U_\alpha$ is simply connected.
Since $f$ is a frontal (see the  appendix),
there exists  a unit normal 
vector field $\nu_\alpha$ 
defined on $U_\alpha$.
Since \eqref{eq:lambda} holds for each
$\lambda=\lambda_\alpha$ on $U_\alpha$,
replacing $\nu_\alpha$ by $-\nu_\alpha$
if necessary, 
we can choose $\nu_\alpha$ satisfying
$
d\hat A=\det(f_{u_\alpha},f_{v_\alpha},\nu_\alpha)du_\alpha\wedge dv_\alpha
$
on $U_\alpha$.
Suppose that $U_\alpha\cap U_\beta$
($\alpha,\beta\in \Lambda$)
is not empty.
Using the chain rule, it can be easily checked that
$$
d\hat A=\det(f_{u_\beta},f_{v_\beta},\nu_\beta)du_\beta\wedge dv_\beta
=\det(f_{u_\alpha},f_{v_\alpha},\nu_\beta)du_\alpha\wedge dv_\alpha
$$
holds on $U_\alpha\cap U_\beta$, 
and so $\nu_\alpha$ coincides with $\nu_\beta$
on $U_\alpha\cap U_\beta$. Thus, there exists a
smooth unit normal vector field $\nu$ on $M^2$ 
satisfying $\nu:=\nu_\alpha$ on each $U_\alpha$.
Therefore, $f$ is co-orientable.
\end{proof}

\begin{Def}\label{def:a2a3}
 A singular point $p$ of a given frontal metric
 is called {\it non-degenerate\/} if its exterior 
derivative
 \[
    d\lambda:=\lambda_u du+\lambda_v dv 
 \]
does not vanish at $p$, 
where $\lambda$ is the function as in
 \eqref{eq:lambda}.
 A frontal metric $d\sigma^2$
 is called a {\it Kossowski metric\/} if
 all of the singular points of the metric are
 non-degenerate.
\end{Def}
As shown in \cite{K}, it holds that
\begin{Lemma}\label{lem:rankone}
 All singular points of a Kossowski metric are of rank $1$.
\end{Lemma}

The following assertion gives
the compatibility between
non-degeneracy of frontal metrics
and that of frontals in $\R^3$.

\begin{Prop}\label{prop:non-deg}
Let $f:M^2\to \R^3$ be a frontal.
Then the singular set of $f$
coincides with that of
the induced metric $d\sigma^2(:=df\cdot df)$.
Moreover, a singular point $p$ of $f$
is non-degenerate as a frontal singularity
$($see the appendix$)$ if and only if
$p$ is a non-degenerate singular point 
of $d\sigma^2$. 
\end{Prop}

\begin{proof}
Compare Definition \ref{def:a2a3}
and the corresponding definition in
the appendix.
\end{proof}

Kossowski \cite{K} proved the following assertion.
(For the sake of reader's convenience, we give the
proof as follows.)

\begin{Thm}[Kossowski \cite{K}]\label{Thm:K1}
 Let $d\sigma^2$ be a co-orientable Kossowski metric.
 Then $K\,d\hat A$ can be smoothly extended as a
 globally defined $2$-form on $M^2$.
\end{Thm}

To prove the assertion, we prepare the following 
lemma, which immediately follows from
the fact that $d\lambda\ne 0$:
\begin{Lemma}\label{lem:lambda}
Let  $(U;u,v)$ be a simply connected local coordinate system 
 centered at a non-degenerate singular point
 $p$ of the frontal metric $d\sigma^2$.
For a $C^\infty$-function $\phi$ on $U$
 which vanishes on the singular set of $d\sigma^2$,
there exist a neighborhood $V(\subset U)$
 of $p$ and  a $C^\infty$-function $\psi$
 on $V$ such that $\phi(u,v)=\lambda(u,v) \psi(u,v)$
 holds on $V$, where
 $\lambda:U\to \R$ is a $C^\infty$-function
 satisfying \eqref{eq:lambda}.
\end{Lemma}
\begin{proof}[Proof of Theorem \ref{Thm:K1}]
 We fix a singular point $p$ of the metric $d\sigma^2$
 arbitrarily.
 Let $\gamma$ be the singular curve passing through $p$. 
 Then one can take an adapted local 
 coordinate system $(U;u,v)$ centered at $p$.
 We set
 \begin{equation}\label{eq:orthonormal}
     \vect{e}_1:=\frac{1}{\sqrt{E}}\partial_u,\qquad
     \vect{e}_2:=\frac{1}{\lambda}\tilde{\vect{e}}_2,
     \quad
     \left(
         \tilde{\vect{e}}_2:=
         -\frac{F}{\sqrt{E}}\partial_u+\sqrt{E}\partial_v
     \right),
 \end{equation}
 which gives an orthonormal frame field
 on $U\setminus \op{Im}(\gamma)$,
where $\op{Im}(\gamma)$ denotes the image 
of the curve $\gamma$.
 Consider a $1$-form
 \begin{equation}\label{eq:omega}
  \omega
   :=-\inner{\nabla_{\partial_u} \vect{e}_1}{\vect{e}_2}\,du
     -
     \inner{\nabla_{\partial_v} \vect{e}_1}{\vect{e}_2}\,dv
 \end{equation}
 defined on $U\setminus \op{Im}(\gamma)$,
 where $\nabla$ is the Levi-Civita connection
 of $d\sigma^2$ on $U\setminus \op{Im}(\gamma)$.
 Using the Kossowski pseudo-connection as in \eqref{eq:Gamma}, 
 we have the following expression
 \[
      \omega=
        -\frac{\Gamma(\partial_u,\vect{e}_1, \tilde{\vect{e}}_2)}
                  {\lambda}\,du
        -
              \frac{\Gamma(\partial_v,\vect{e}_1, \tilde{\vect{e}}_2)}
                  {\lambda}\,dv.
 \]
 The vector field  $\tilde{\vect{e}}_2$ 
 is a smooth vector field on $U$ which vanishes along $\gamma$.
 Since $d\sigma^2$ is an admissible metric, 
 $\Gamma(\partial_u,\vect{e}_1, \tilde{\vect e}_2)$ 
 and 
 $\Gamma(\partial_v,\vect{e}_1, \tilde{\vect e}_2)$
 vanish on $\gamma$.
 By Lemma \ref{lem:lambda}, there exist 
 two locally defined smooth functions $a$, $b$ 
 such that
 \[
    \Gamma(\partial_u,\vect{e}_1, \tilde {\vect e}_2)=
     \lambda(u,v) a(u,v),\qquad
    \Gamma(\partial_v,\vect{e}_1, \tilde {\vect e}_2)=
     \lambda(u,v) b(u,v).
 \]
 Thus we can write
 \[
    \omega=-a(u,v)\,du-b(u,v)\,dv,
 \]
 which implies that $\omega$ can be extended as
 a smooth $1$-form on $U$.
 Since $d \lambda \ne 0$,
 the function $\lambda$ changes sign at 
 the singular curve $\gamma$.
 Since $d\sigma^2$ is co-orientable on a simply connected domain,
 the following two subsets 
 \[
     U_+:=\{p\in M^2\,;\, dA_p=d\hat A_p\},\qquad
     U_-:=\{p\in M^2\,;\, dA_p=-d\hat A_p\}
 \]
of $M^2$ are defined.
 Since $\{\vect{e}_1,\vect{e}_2\}$ is a positive (resp.\ negative)
 frame on $U_+$
  (resp.\ on $U_-$),
 the classical connection theory
 yields that $d\omega$ coincides with $K\, dA$ 
 (resp.\ $-K\,dA$)
 on $U_+$
 (resp.\ on $U_-$).
 Thus $d\omega=K\,d\hat A$ holds
 on $U\setminus \op{Im}(\gamma)$. 
 Then by continuity,
 the identity $d\omega=K\,d\hat A$
 holds on  $U$.
 Since $\omega$ is a smooth $1$-form, 
 we get the assertion.
\end{proof}

Let $\nu$ be the unit normal vector field
of a frontal $f:U\to \R^3$.
As pointed out in \cite{MSUY},
 $K\,d\hat A$ coincides with 
the pull-back of the 
canonical area element of the 
unit sphere $S^2$ by $\nu$.
So $f$ is a wave front if $K\,d\hat A$
does not vanish on $U$.

A frontal metric $d\sigma^2$
on a real analytic manifold 
is called a {\it real analytic Kossowski metric}
if one can take $E,F,G,\lambda$ to be real analytic
functions on each real analytic local
coordinate system $(U;u,v)$ 
in Definition \ref{def:frontal}. 
Kossowski 
proved the following:

\begin{Fact}[Kossowski \cite{K}]\label{fact:K}
 Let $p\in M^2$ be a singular point
 of a real analytic Kossowski metric $d\sigma^2$.
 If $Kd\hat A$ does not vanish at $p$,
 there exist a neighborhood $U$ of $p$
 and a real analytic wave front
 $f:U\to \R^3$ such that
 the first fundamental form of $f$
 coincides with $d\sigma^2$ on $U$.
\end{Fact}

By this realization theorem, 
it is reasonable to see
Kossowski metrics
as the best class of metrics to 
describe intrinsic invariants
on wave fronts. 
Intrinsic properties of cuspidal edges
and swallowtails are not discussed
in \cite{K}. From now on, we shall
give  intrinsic characterizations of
cuspidal edges and swallowtails
(cf.\ Figure~\ref{fig:sw}).

\medskip
Let $d\sigma^2$ be a Kossowski metric
and  $p$ a non-degenerate singular point.
Let $(u,v)$ be a local coordinate system
centered at $p$.
By the implicit function theorem,
there exists a regular curve $\gamma(t)$
($|t|<\epsilon$)
on the $uv$-plane such that $\gamma(0)=p$,
where  $\epsilon>0$.
In this setting, there exists a smooth 
vector field $\eta(t)$ along $\gamma(t)$
such that $\eta(t)$ belongs to $\N_{\gamma(t)}$.

\begin{Def} \label{def:a2a3add}
If $\eta(0)$ is linearly independent
of the singular direction $\dot \gamma(0)$,
then $p$ is called an {\it $A_2$-point}.
If $p$ is not an $A_2$-point, but
\[
   \left.\frac{d}{dt}\det\bigl(\dot \gamma(t),\eta(t)\bigr)\right|_{t=0}\ne 0
\]
holds, then $p$ is called an {\it $A_3$-point}, where
$\det\bigl(\dot \gamma(t),\eta(t)\bigr)$ is the determinant
of two vectors $\dot \gamma(t),\eta(t)$ in the $uv$-plane $\R^2$.
\end{Def}

The following assertion holds:

\begin{figure}
 \begin{center}
   \begin{tabular}{c@{\hspace{2cm}}c}
        \includegraphics[height=4.0cm]{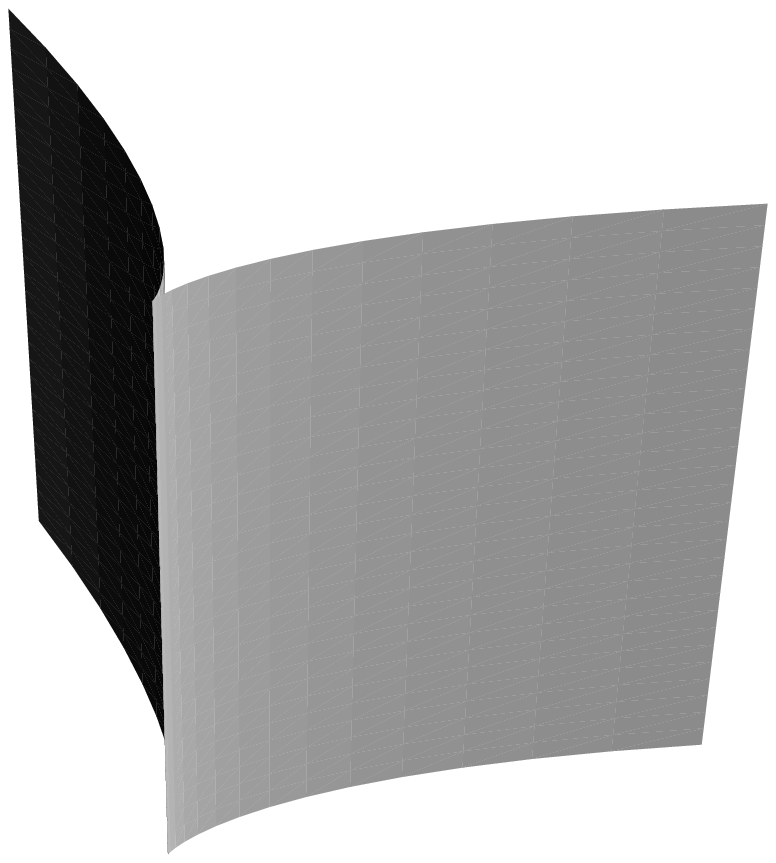} &
        \includegraphics[height=4.0cm]{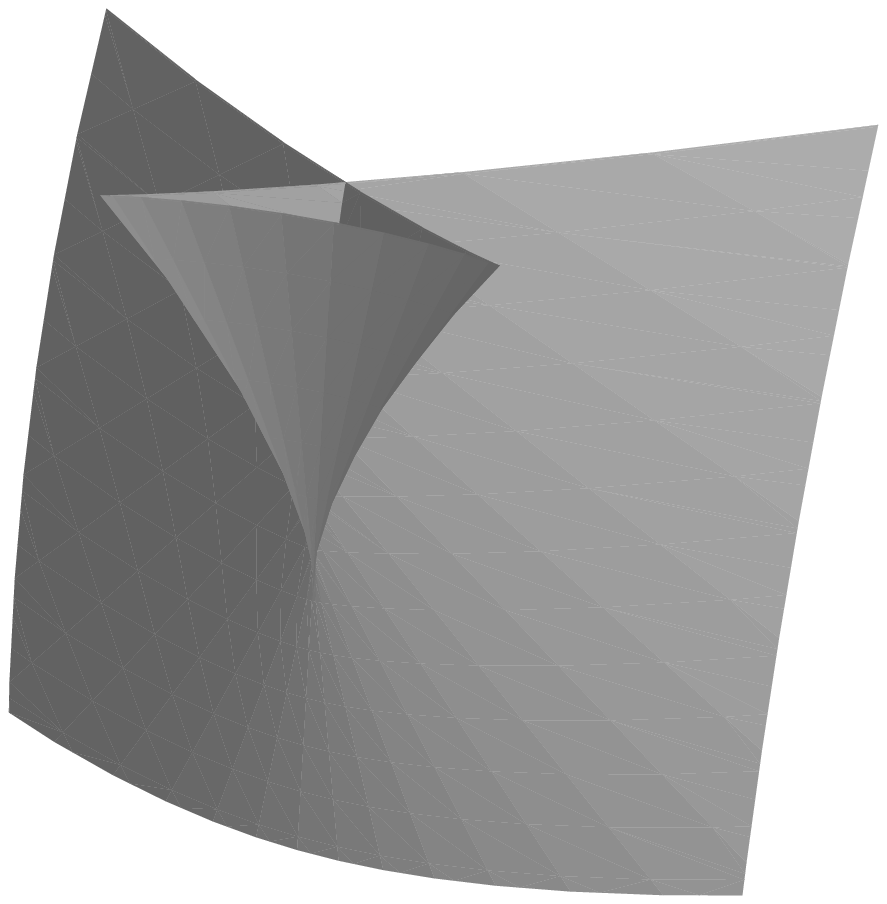}
    \end{tabular}
 \end{center}
\caption{A cuspidal edge and a swallowtail}
\label{fig:sw}
\end{figure}

\begin{Prop}\label{prop:a2a3}
Let $f:M^2\to \R^3$ be a $C^\infty$ wave front.
If $f$ has
a cuspidal edge  
 {\rm(}resp.\ a swallowtail{\rm)}
 singular point\footnote{The
definitions of cuspidal edges and swallowtails
are given in the appendix.}, then it
corresponds to an $A_2$-point 
{\rm(}resp.\ an  $A_3$-point{\rm)}
of
the induced metric $d\sigma^2(:=df\cdot df)$. 
Conversely, 
if a germ of a
Kossowski metric
at an $A_2$-point 
$($resp.\ an $A_3$-point$)$ 
is real analytic 
and $Kd\hat A$ does not vanish, 
then it can be realized as the induced metric
of a wave front with
cuspidal edges $($resp. a swallowtail$)$.
\end{Prop}

\begin{proof}
Comparing Definition \ref{def:a2a3add}
and Fact~\ref{fact:krsuy} in the appendix,
we get the first assertion.
Finally, we get the second assertion
applying Fact \ref{fact:K}.
\end{proof}

\begin{Rmk}
 If $f$ is not a wave front, singular points of
 $f$ corresponding to $A_2$-points
(resp.~$A_3$-points)
of 
the first fundamental
 form might not be cuspidal edges 
(resp.~swallowtails).
 In fact, a cuspidal cross cap $f_{CCR}$
given in the appendix
(resp.~a $C^\infty$ map $f_2$ given in 
\cite[Example~1.6]{MSUY})
is a frontal (but not a front) which 
induces a Kossowski metric
with an $A_2$-point 
(resp.~an $A_3$-point) 
satisfying $Kd\hat A=0$.
Kossowski metrics
might have singular points other than
$A_2$ or $A_3$
in general
(cf.~the induced metric of
the $C^\infty$ map $f_1$ given in 
\cite[Example~1.6]{MSUY}).
\end{Rmk}

In \cite{MSUY}, the limiting normal 
curvature $\kappa_\nu$ is introduced for 
non-degenerate singular points
of wave fronts, which can be
interpreted as the 
normal curvature of the surface 
with respect to the singular direction.
Moreover, in \cite{MSUY}, 
the cuspidal curvature $\kappa_c$
along the cuspidal edge singularities
was also defined, 
and it was also shown that the product
$\kappa_\Pi:=\kappa_\nu\kappa_c$ is an intrinsic 
invariant of cuspidal edges.
Isometric deformations of cuspidal edges
were discussed in \cite{NUY}
and it was shown that
$\kappa_\nu$ and 
$\kappa_c$ are both extrinsic invariants.
The condition
$K\, d\hat A\ne 0$
(cf.\ Fact \ref{fact:K})
is equivalent to the condition that
$\kappa_\nu \ne 0$ at the singular points.
On the other hand,
the singular curvature $\kappa_s$
along the cuspidal edge singularities
was defined in \cite{SUY}, 
which is an intrinsic invariant, and
played an important role in describing the Gauss-Bonnet type formula
for closed wave fronts. 
From now on, we shall explain these two
intrinsic invariants $\kappa_s$ and
$\kappa_\Pi$ of cuspidal edge singularities
in terms of Kossowski metrics.

\begin{Def}\label{def:adapted}
 Let $p$ be an $A_2$-point of
 a given Kossowski metric $d\sigma^2$.
 Then an adapted (local) coordinate system $(u,v)$
 of $M^2$ in the sense of Definition~\ref{def:adms-coord} is called
 a {\it strongly adapted coordinate system\/}
 if the $u$-axis consists of singular points.
 (By the adaptedness,
 $\partial_v\in \N_{(u,0)}$ holds.)
\end{Def}

The existence of an strongly adapted coordinate system
at a given $A_2$-point
can be proven easily. 
Since the strongly adapted coordinate system 
satisfies the property in the 
assumption of Corollary~\ref{cor:property},
the following assertion is proved: 
\begin{Prop}\label{prop:adapted}
 Let $(u,v)$ be a strongly adapted coordinate system.
 Then it holds that
 \begin{equation}\label{eq:EFG}
  F(u,0)=G(u,0)=E_v(u,0)=G_v(u,0)=0,
 \end{equation}
 where $d\sigma^2=E\,du^2+2F\,du\,dv+G\,dv^2$.
 Moreover, if another local coordinate system
 $(\xi,\eta)$ satisfies
 \begin{equation}\label{eq:EFG2}
v(\xi,0)=u_\eta(\xi,0)=0,\quad
   u_\xi(\xi,0)\ne 0,\quad  v_\eta(\xi,0)\ne 0,
 \end{equation}
 then $(\xi,\eta)$ is also a strongly adapted coordinate system.
\end{Prop}

Let $p$ be an $A_2$-point of
a given Kossowski metric $d\sigma^2$,
and $(u,v)$ a strongly adapted coordinate system
centered at $p$. Without loss of generality,
we may assume that $\lambda_v>0$, 
where $\lambda$ is a function satisfying \eqref{eq:lambda}.
We set
\begin{equation}\label{eq:ks}
 \kappa_s(u):=
  \frac{-F_vE_u + 2EF_{uv}-EE_{vv}}
  {2E^{3/2}\lambda_v},
\end{equation}
which is called the {\it singular curvature}%
\footnote{There is a typographical error in \cite[Proposition 1.8 in Page 497]{SUY}.
  In fact,  
  the right-hand side of the expression of
  $\kappa_s(u)$ should be divided by 2.
}
at the singular point $(u,0)$,
where $d\sigma^2=E\,du^2+2F\,du\,dv+G\,dv^2$.
As shown in \cite{SUY},
the singular curvature along cuspidal edges
has the same expression as \eqref{eq:ks}.
So the above definition gives a generalization 
of singular curvature
for $A_2$-points of Kossowski metrics. 
The following assertion holds:

\begin{Prop}\label{prop:int}
 The value of $\kappa_s$ does not depend
 on a choice of strongly adapted coordinate systems satisfying 
 $\lambda_v>0$.
 In particular, it does not depend on the
 orientation of the singular curve. 
\end{Prop}
\begin{proof}
 We let $(\xi,\eta)$ be another
strongly adapted coordinate system. 
 Then it holds that
 \begin{align}
  \tilde E&=E u_\xi^2+2F u_\xi v_\xi+G v_\xi^2, \label{eq:ee}\\
  \tilde F&=E u_\xi u_\eta+F (u_\xi v_\eta+u_\eta v_\xi)
  +G u_\xi u_\eta, \label{eq:ff}
  \\
  \tilde G&=E u_\eta^2+2F u_\eta v_\eta+G v_\eta^2, 
  \label{eq:gg}
 \end{align}
 where
 $d\sigma^2=\tilde E\,d\xi^2+2\tilde F\,d\xi\, d\eta+\tilde G\,d\eta^2$.
 Using \eqref{eq:EFG} and \eqref{eq:EFG2}, we have that
 \begin{align*}
  \tilde E&=E u_\xi^2,\qquad  
  \tilde E_{\xi}=u_\xi^3 E_u+2 Eu_{\xi\xi} u_\xi, \\
  \tilde E_{\eta\eta}&=
  u_\xi^2 v_\eta^2 E_{vv}+u_{\eta\eta} u_\xi^2 E_u
  +2 Eu_{\xi} u_{\xi\eta\eta}
  +4 u_\xi v_\eta v_{\xi\eta} F_v, \\
  \tilde F_\eta&
  =u_{\eta\eta}u_\xi E+u_\xi v_\eta^2 F_v, \\
  \tilde F_{\xi\eta}
  &=u_{\eta\eta} u_\xi^2 E_u
  +u_{\eta\eta}u_{\xi\xi} E+u_\xi u_{\xi\eta\eta} E+v_\eta^2 u_\xi^2 F_{uv}
  +v_\eta^2u_{\xi\xi} F_v+2 u_\xi v_{\xi\eta} v_\eta F_v
 \end{align*}
hold along the singular curve.
 Using these relations, one can see that
 \[
    -\tilde F_\eta\tilde E_\xi + 2\tilde E\tilde F_{\xi\eta}-
     \tilde E \tilde E_{\eta\eta}
     =u_\xi^4 v_\eta^2\left(-F_vE_u + 2EF_{uv}-EE_{vv}\right)
 \]
 holds on the $u$-axis.
 Since $d\sigma^2$ is a frontal metric,
 there exists a $C^\infty$-function $\tilde \lambda$
 such that $\tilde E \tilde G-\tilde F^2=\tilde \lambda^2$.
 Since
$\tilde \lambda=\pm (u_\xi v_\eta -u_\eta v_\xi) \lambda$,
 the fact that $\lambda(u,0)=0$ implies that
 \begin{equation}\label{eq:c-change}
  \tilde \lambda_\eta=\pm v^2_\eta u_\xi \lambda_v
 \end{equation}
holds on the singular curve.
 Replacing $(\xi,\eta)$ by $(-\xi,\eta)$ if necessary,
 we may assume that $u_\xi>0$.
 If we assume $\tilde \lambda_\eta>0$, then  
 $\tilde \lambda_\eta=v^2_\eta u_\xi \lambda_v$
 holds.
 Using the relation $\tilde E^{3/2}=u_\xi^3 E^{3/2}$,
 one can easily check the coordinate independence of
 the definition of $\kappa_s$.
 The last assertion follows if we consider
 the coordinate change $(u,v)\mapsto (-u,v)$
 (in this case, $F$ and $u$ change to $-F$ and $-u$, respectively).
\end{proof}

\begin{Def}\label{def:st}
 A strongly adapted coordinate system $(u,v)$ at an
 $A_2$-point $p$
 is said to be {\it normalized\/}
 if it satisfies the following three conditions:
 \begin{itemize}
  \item[{\rm (i)}] $E(u,0)=1$,
  \item[{\rm (ii)}] $\lambda_v(u,0)=\pm 1$, that is,  $G_{vv}(u,0)=2$,
  \item[{\rm (iii)}] $F(u,v)=0$,
 \end{itemize}
 where $E$, $F$, $G$, $\lambda$ are smooth functions
 given by $d\sigma^2=E\,du^2+2\,F\,du\,dv+G\,dv^2$
 and $EG-F^2=\lambda^2$. 
\end{Def}
\begin{Prop}\label{prop:Kos}
 Let $p$ be an $A_2$-point 
 of a Kossowski metric $d\sigma^2$.
 Then there exists a normalized strongly adapted coordinate
 system at $p$.
\end{Prop}
\begin{proof}
 We fix a strongly adapted coordinate system
 $(U;a,b)$ at $p$, and
 let $d\sigma^2=\check E\,da^2+2 \check F\,da\,db+\check G\,db^2$.
 Then
 \[
    X:=\partial_{a},\qquad
    Y:=-\check F\partial_{a}+\check E \partial_{b}
 \]
 are two vector fields on $U$
 that are mutually orthogonal.
 Then by applying the lemma in Page 182 just 
 after Proposition 5.2 in Kobayashi-Nomizu \cite{KN},
 there exists a local coordinate system
 $(x,y)$ such that
 $\partial_x$ and $\partial_y$
 are proportional to $X$ and $Y$,
 respectively, and
$y(a,0)=0$ (namely, the singular set
is the $x$-axis).
Moreover, since $\partial_b$ is the null direction on the
singular set $\{b=0\}$, 
$Y$ gives a null vector field along the singular set.
Thus $(x,y)$ is a strongly adapted coordinate system.
Since $X$, $Y$ are orthogonal, 
the metric has the expression
\[
     d\sigma^2=\hat E\,dx^2+\hat G\,dy^2,
\]
 where $\hat E (>0)$  and $\hat G$ are smooth functions in $(x,y)$.
 Consider the coordinate change
 \[
     \xi:=\int_0^x\sqrt{\hat E(t,0)}\,dt,\qquad \eta:=y,
 \]
 which is strongly adapted, and the metric 
 can be expressed by
 $d\sigma^2=\tilde E\,d\xi^2+\tilde G\,d\eta^2$
 with
 \[
     \tilde E(\xi,0)=1.
 \]
 In particular, $\tilde G_{\eta\eta}>0$ holds on the $\xi$-axis.
 In fact,
 since $d\sigma^2$ is frontal, there exists
 a smooth function $\tilde\lambda$ such that 
 $\tilde\lambda^2=\tilde E\tilde G$.
 Since $\tilde\lambda=0$ on the singular set $\{\eta=0\}$,
 non-degeneracy (cf.\ Definition \ref{def:a2a3})
 implies $\tilde\lambda_{\eta}\neq 0$ on the $\xi$-axis.
 Differentiating $\tilde\lambda^2 = \tilde E\tilde G$ 
 twice with respect to $\eta$, we have
 \[
     2\lambda\lambda_{\eta\eta}+2(\lambda_{\eta})^2
     = 2 \tilde E_{\eta}\tilde G_{\eta}+ \tilde E_{\eta\eta}\tilde G
       + \tilde E\tilde G_{\eta\eta}.
 \]
 Then by Proposition~\ref{prop:property} and
 Corollary~\ref{cor:property},
 it holds on the $\xi$-axis that
 \[
     \tilde G_{\eta\eta}=2(\lambda_{\eta})^2>0,
 \]
 and hence $\tilde G_{\eta\eta}(\xi,0)>0$.
 Now we set
 \[
    u:=\xi, \qquad
    v:=\frac{\eta}{\sqrt[4]{\tilde G_{\eta\eta}(\xi,0)/2}},
 \]
 giving a desired normalized strongly adapted coordinate system
 as follows:
 By definition,
 we have the expression
 $d\sigma^2=E\,du^2+G\,dv^2$.
 It is obvious that $\partial_v$ is perpendicular to $\partial_u$.
 Differentiating \eqref{eq:gg},
 we get
 \[
     \tilde G_{\eta\eta}= G_{vv} v_\eta^4,
 \]
 where
 we used the facts
 $u_\eta=F_\eta=0$ on the $\xi$-axis.
 So 
 $G_{vv}=2$ holds on the singular curve.
 Differentiating 
 $\lambda^2=E G-F^2=EG$
 by $v$ twice, and 
 using the fact that $G_v(u,0)=0$,
 we get 
 $\lambda_v(u,0)=1$,
proving the assertion.
\end{proof}

\begin{Def}\label{rmk:U}
Let $\{(U_\alpha;u_\alpha,v_\alpha,\lambda_\alpha)\}_{\alpha\in \Lambda}$
be a family of quadruple 
satisfying \eqref{eq:star}.
An adapted coordinate system 
$(u,v):=(u_\alpha,v_\alpha)$
centered at a singular point $p$
is said to be {\it compatible} with respect to
the co-orientation of $d\sigma^2$ if
$\partial \lambda_\alpha/\partial v_\alpha$
is positive at $p$.
\end{Def}

\medskip
Let $(u,v)$ and $(\xi,\eta)$ be two normalized strongly
adapted coordinate systems at an $A_2$-point $p$.
Then the property $\lambda_v(u,0)=1$
yields that
$v(\xi,0)=0$ and 
$v_\eta(\xi,0)=1$. 
Hence if the limit
\[
   \dy\lim_{v\to 0}v K(u,v)
\]
exists, it does not depend on a choice of such $(u,v)$
up to $\pm$-ambiguity,
where $K$ is the Gaussian curvature of $d\sigma^2$.
The following assertion holds:

\begin{Prop}\label{prop:lim}
 Let $(U;u,v)$ be a normalized strongly 
adapted coordinate system
 at an $A_2$-point of a given
 Kossowski metric $d\sigma^2$.
 Then the limit
 \[
    \tilde \kappa_\Pi:=\dy\lim_{v\to 0}v K(u,v)
 \]
 exists, whose absolute valued
$|\tilde \kappa_\Pi|$ 
does not depend on a choice of such $(u,v)$.
Moreover, if $d\sigma^2$ is co-oriented
and $(U;u,v)$ is compatible with respect to
the co-orientation of $d\sigma^2$,
then $\tilde \kappa_\Pi$ itself is 
determined without $\pm$-ambiguity.
\end{Prop}

\begin{proof}
 By Theorem \ref{Thm:K1},
 $K\,\lambda du\wedge dv=K\, d\hat A$
 is a smooth $2$-from on $M^2$,
 and thus
 $K\,\lambda$ is a $C^\infty$-function
 on $U$.
 Since
 $K\lambda=(v K) (\lambda/v)$,
 the facts $\lambda(u,0)=0$ 
 and $\lambda_v=1$ yield that
 $\lambda/v$  is a non-vanishing smooth
 function near the $u$-axis.
 Thus  $vK(u,v)$ is also
 a smooth function near the $u$-axis, which proves the assertion.
\end{proof}

By \cite[(2.16)]{MSUY}, we get the following assertion,
which is a refinement of \cite[Theorem 2.17]{MSUY}.
\begin{Cor} 
 Let $f:M^2\to \R^3$ be a wave front, and
 $p\in M^2$ a cuspidal edge singular point of
 $f$. Then the absolute value of $\tilde \kappa_\Pi$
 at $p$ as an $A_2$-point with
 respect to the first fundamental form of $f$
 coincides with that of
 product curvature $\kappa_{\Pi}$ 
 at $p$ defined in \cite{MSUY}.
\end{Cor}

As an application of the existence of 
normalized strongly adapted coordinate systems at
$A_2$-points, we can give the following characterization of
Kossowski metrics at singularities:
\begin{Prop}\label{prop:ab}
 Let $p$ be an $A_2$-point of
 a Kossowski metric $d\sigma^2$.
 Then there exists a local coordinate system $(u,v)$ centered at $p$ 
 and smooth function germs
 $\alpha$, $\beta$ at $p$ so that
 \begin{equation}\label{eq:ds}
  d\sigma^2=\bigl(1+v^2\alpha\bigr)du^2+v^2\bigl(1+v\beta\bigr)dv^2.
 \end{equation}
 Conversely, any metrics described as in \eqref{eq:ds}
 give germs of Kossowski metrics
having $A_2$-points at the origin.
\end{Prop}
\begin{proof}
 Let $(u,v)$ be a normalized strongly adapted coordinate system 
 at $p$ having the expression
 $d\sigma^2=E\,du^2+G\,dv^2$.
 Since $E=1$ holds on the $u$-axis,
 there exists a $C^\infty$-function germ $A$
 at $p$ such that
 $E(u,v)-1=v A(u,v)$.
 By differentiating it,
 \[
    E_v=(E-1)_v=vA_v+A
 \]
 holds.
 Since $E_v=0$ holds along the $u$-axis,
 $A=E_v-vA_v$ also vanishes
 on the $u$-axis. 
 By Lemma \ref{lem:lambda},  we can write
 $A=v\alpha$, where
 $\alpha$ is a $C^\infty$-function germ  at $p$. 
 So we get the expression
 $E=1+v^2 \alpha(u,v)$.
 On the other hand,
 by using the relations
 $G=G_v=0$, 
 there exists a $C^\infty$-function germ
 $B$ such that $G=v^2B(u,v)$.
 Since $G_{vv}(u,0)=2$, 
 we can write
 $B-1=v\beta(u,v)$, where
 $\beta$ is a $C^\infty$-function germ,
 and get the expression
 \[
    G=v^2B(u,v)=v^2\bigl(1+v\beta(u,v)\bigr),
 \]
which proves the first assertion.
 The second assertion can be proved easily.
\end{proof}
\begin{Rmk}\label{rmk:expression}
 Under the expression \eqref{eq:ds} 
 of the Kossowski metric at an $A_2$-point,
 the singular curvature $\kappa_s$ and the product 
 curvature $\tilde \kappa_\Pi$ at $p$ are given by
 \begin{align}
  \kappa_s&=-\alpha(u,0), \\
{\tilde\kappa}_\Pi&=\frac{\alpha(u,0) \beta(u,0)-3 \alpha_v(u,0)}{2}.
  \label{eq:Pi}
 \end{align}
\end{Rmk}

Using \eqref{eq:Pi}, we get the following assertion.
\begin{Cor}[An intrinsic characterization of
 cuspidal edges]\label{cor:chr}
 Let $p$ be a cuspidal edge singular point of
 a wave front $f:M^2\to \R^3$, 
 whose limiting normal curvature $\kappa_\nu$
 does not vanish at $p$.
 Then there exists a local coordinate system $(u,v)$
 centered at $p$
 such that the first fundamental form of $f$
 has the expression \eqref{eq:ds}.
 Conversely, if $\alpha,\beta$ are two
 real analytic function germs,
 then the metric $d\sigma^2$ given by
 \eqref{eq:ds}
 can be realized as a first fundamental form
 of a real analytic wave front in
 $\R^3$
 under the assumption that
 \begin{equation}
  \alpha \beta-3 \alpha_v\ne 0.
 \end{equation}
\end{Cor}

\begin{proof}
In \cite{MSUY},
it was shown that
$\kappa_\Pi\ne 0$ is equivalent to
the condition $K d\hat A\ne 0$.
So Fact \ref{fact:K} and
\eqref{eq:Pi} give the conclusion.
\end{proof}

A refinement of Corollary \ref{cor:chr}
is given in \cite{NUY}, where
the ambiguity of such a realization
is discussed and a normalization theorem
of generic cuspidal edges 
is given by the use of this ambiguity.

\section{Gauss-Bonnet formulas for Kossowski metrics}
\label{sec:GB}

Let $M^2$ be an oriented $2$-manifold.
A vector bundle $\E$ of rank $2$
with a metric $\inner{~}{~}$ and a metric connection
$D$ is called a {\em coherent tangent bundle\/}
if there is a bundle homomorphism
\[
   \psi\colon{}TM^2\longrightarrow \E
\]
such that
\begin{equation}\label{eq:c}
    D_{X}\psi(Y)-D_{Y}\psi(X)=\psi([X,Y]) 
\end{equation}
holds for all vector fields $X,Y$ on $M^2$ 
(cf.\ \cite{SUY2} and \cite{SUY3}).

In this setting, 
the pull-back of the metric $d\sigma^2:=\psi^*\inner{~}{~}$ is called 
the {\em first fundamental form\/} of $\psi$.
A point $p\in M^2$ is called a {\it singular point\/} 
(of $\psi$) if $\psi_p\colon{}T_pM^2\to\E_p$ is not a bijection,
where $\E_p$ is the fiber of $\E$ at $p$. 
The singular points of $\psi$ are the singular points of $d\sigma^2$.

The vector bundle $\E$ is called {\it orientable\/} 
if there exists a smooth non-vanishing skew-symmetric bilinear section 
$\mu$ of $M^2$ into $\E^*\wedge\E^*$
such that $\mu(e_1,e_2)=\pm 1$ for any orthonormal frame $\{e_1,e_2\}$
on $\E$, where $\E^*\wedge\E^*$
denotes the determinant line bundle of the dual bundle $\E^*$.
The form $\mu$ is uniquely  determined up to $\pm$-ambiguity.
An {\em orientation\/} of the coherent tangent bundle $\E$ is a choice
of $\mu$.
A frame $\{e_1,e_2\}$ is called {\em positive\/} 
with respect to the orientation $\mu$
if $\mu(e_1,e_2)=1$.

\begin{Thm}\label{thm:coherent}
 Let $d\sigma^2$ be a Kossowski metric
 defined on a $2$-manifold $M^2$ without boundary.
 Then there exists a coherent tangent bundle
 \[
    \psi:TM^2\longrightarrow (\E, \inner{~}{~}, D)
 \]
 such that the first fundamental form
 induced by $\psi$ coincides with
 $d\sigma^2$.
 Moreover, $\E$ is orientable if
 $d\sigma^2$ is co-orientable.
\end{Thm}
\begin{proof}
 Let 
 $\{(U_\alpha;u_\alpha,v_\alpha)\}_{\alpha\in \Lambda}$
 be a covering of $M^2$ consisting of
 local adapted coordinate systems which are compatible
 with respect to the orientation of $M^2$.
 Since $d\sigma^2$ is a Kossowski metric,
 there exists a $C^\infty$-function 
 $\lambda_\alpha$ on $U_\alpha$ ($\alpha\in \Lambda$)
 such that
 \[
     E_\alpha G_\alpha-(F_\alpha)^2=(\lambda_\alpha)^2,
 \]
 where 
 $d\sigma^2=E_\alpha du_\alpha^2+2F_\alpha du_\alpha dv_\alpha+G_\alpha
 dv^2_\alpha$ on $U_\alpha$.

 We fix two indices $\alpha,\beta\in \Lambda$
 so that $U_\alpha\cap U_\beta\ne \emptyset$, and
 set
 \[
     d\sigma^2=E\,du^2+2F\,du\,dv+G\,dv^2
              =\tilde E\,d\xi^2+2\tilde F\,d\xi\,d\eta
              +\tilde G\, d \eta^2,
 \]
 where we set
\[
    (U;u,v):=(U_\alpha;u_\alpha,v_\alpha),\qquad
    (V;\xi,\eta):=(U_\beta;u_\beta,v_\beta)
 \]
for the sake of simplicity.
 We now set
 \begin{align*}
    (\vect{e}_1,\vect{e}_2)
      &=(\partial_u,\partial_v)\T_\alpha,
      \qquad
      \T_\alpha:=\pmt{1/\sqrt{E} & -F/(\lambda\sqrt{E})\\
                          0      &  \sqrt{E}/\lambda},\\
    (\tilde{\vect{e}}_1,\tilde{\vect{e}}_2)
      &=(\partial_{\xi},\partial_{\eta})\T_{\beta},
     \qquad
     \T_\beta:=
     \pmt{1/\sqrt{\tilde E} & -\tilde F/(\tilde \lambda\sqrt{\tilde E})\\
         0      &  \sqrt{\tilde E}/\tilde \lambda},
 \end{align*}
 where
 $\lambda:=\lambda_\alpha$ and $\tilde\lambda:=\lambda_\beta$
(cf.\ \eqref{eq:orthonormal}).
 Then $(\vect{e}_1,\vect{e}_2)$ and
 $(\tilde{\vect{e}}_1,\tilde{\vect{e}}_2)$ 
 are orthonormal frame fields
 on $U\setminus \Sigma$ and $V\setminus \Sigma$
 respectively, where $\Sigma$ denotes the singular set of
 the metric $d\sigma^2$ on $M^2$.
 It holds on $(U\cap V)\setminus \Sigma$
 that
 \begin{equation}\label{ee}
 (\tilde{\vect{e}}_1,\tilde{\vect{e}}_2)=
  (\vect{e}_1,\vect{e}_2)\T_\alpha^{-1} \J_{\alpha\beta}
                \T_\beta,
         \qquad
	 \J_{\alpha\beta}:=
	 \pmt{u_\xi & u_\eta \\ v_\xi & v_\eta}.
 \end{equation}
 In particular,
 \begin{equation}\label{gab}
    g_{\alpha\beta}:=
      \T_\alpha^{-1} \J_{\alpha\beta}\T_\beta
 \end{equation}
 can be considered as a matrix valued function
 defined on $(U\cap V)\setminus \Sigma$
 which takes values in the orthogonal group
 $\O(2)$.
 Since two local coordinate systems $(u,v)$ and
$(\xi,\eta)$ are adapted,
 \eqref{gab} can be reduced to 
 \[
 \sqrt{E\tilde E}\,g_{\alpha\beta}
 =
 \left(
 \begin{array}{cc}
  E u_\xi+F v_\xi & 
   -E (\tilde F/\tilde \lambda) u_\xi+
   E \tilde E (u_\eta/\tilde \lambda)-F (\tilde F/\tilde \lambda) v_\xi
   +(F/\tilde \lambda) \tilde E v_\eta \\
  v_\xi \lambda  & (\tilde E v_\eta-\tilde F v_\xi) (\lambda/\tilde \lambda) \\
 \end{array}
 \right).
 \]
 Here, by \eqref{eq:u2} and Lemma~\ref{lem:lambda},
 \[
    \tilde F/\tilde \lambda,\quad u_\eta/\tilde \lambda,\quad F/\tilde \lambda,
    \quad \lambda/\tilde \lambda
 \]
 are smooth functions on $U\cap V$.
 Since $E\tilde E>0$ on $U\cap V$, we can conclude that
 $g_{\alpha\beta}$ can be extended as a smooth map
 \[
    g_{\alpha\beta}:U_\alpha\cap U_\beta\to \O(2). 
 \]
 Since $g_{\alpha\beta}$ is a transition function
 of the restriction of vector bundle $TM^2$
 into $M^2\setminus \Sigma$, the co-cycle condition
 \begin{equation}\label{eq:cocycle}
  g_{\alpha\beta}g_{\beta\gamma}g_{\gamma\alpha}=\id
 \end{equation}
 holds on $M^2\setminus \Sigma$.
 By the continuity, \eqref{eq:cocycle} holds on the
 whole of $M^2$. Thus, there exists a vector bundle $\mathcal E$
 with inner product $\inner{~}{~}$
 whose transition functions are $\{g_{\alpha\beta}\}$,
 namely, there exist smooth orthonormal frame fields
 $\Gamma_\alpha$  of $\E$
 ($\alpha\in \Lambda$)
 satisfying
 \begin{equation}\label{gamma0}
  \Gamma_\beta=\Gamma_\alpha g_{\alpha\beta}
   \qquad (\alpha,\beta\in \Lambda).
 \end{equation}
 By \eqref{eq:omega},
 we can get a smooth $1$-form 
 $\omega_{\alpha}:= \omega$ on $U_\alpha$.
 Since $\pmt{0 & \omega_\alpha \\ -\omega_\alpha & 0}$
 is a usual connection form of the
 Levi-Civita connection of $d\sigma^2$,
 the identity
 \begin{equation}\label{eq:conn}
  \pmt{0 & \omega_\beta \\ -\omega_\beta & 0}=
   g_{\alpha\beta}^{-1}(dg _{\alpha\beta})
   +g_{\alpha\beta}^{-1}\pmt{%
   0 & \omega_\alpha \\ -\omega_\alpha & 0}
   g_{\alpha\beta}
 \end{equation}
 holds on $(U\cap V)\setminus \Sigma$.
 Then by the continuity of $\{\omega_\alpha\}$ and $\{g_{\alpha\beta}\}$,
 \eqref{eq:conn} holds on $U\cap V$.
 For each $\alpha$, we set
 \[
    D_{w}\hat e_1:=-\omega_\alpha(w)\hat e_2,\qquad
    D_{w}\hat e_2:=\omega_\alpha(w)\hat e_1\qquad (w\in TU),
 \]
 where $\Gamma_\alpha=(\hat e_1,\hat e_2)$.
 By \eqref{eq:conn}, $D$ gives a globally defined
 metric connection of $\E$.

 We now define a bundle homomorphism
 $\psi:TM^2|_{M^2\setminus \Sigma}\to \E|_{M^2\setminus \Sigma}$
 by
 \[
   \psi(\vect{e}_1)=\hat e_1,\qquad  
   \psi(\vect{e}_2)=\hat e_2,
 \]
 where $TM^2|_{M^2\setminus \Sigma}$
 (resp.\ $\E|_{M^2\setminus \Sigma}$)
 is the restriction $TM^2$ (resp.\ $\E$)
 to $M^2\setminus \Sigma$.
 By \eqref{ee}, \eqref{gab} and \eqref{gamma0},
 it can be easily checked that the definition of 
$\psi$ does not depend
 on the choice of the index  $\alpha$.
 Moreover, the definition of $\psi$ yields that
 $(\psi(\partial_u),\psi(\partial_v))=\Gamma_\alpha
 \T_{\alpha}^{-1}$.
 Since
 \[
     \T^{-1}_{\alpha}=\pmt{\sqrt{E} & F/\sqrt{E} \\
                        0     & \lambda/\sqrt{E}}
 \]
 is a matrix-valued $C^\infty$-function on $U=U_\alpha$,
$\psi$ can be smoothly extended as a bundle homomorphism
 $\psi:TM^2\to \E$.
The transition functions
 $g_{\alpha\beta}$ and the connection forms
$\omega_{\alpha}$ are common in
two vector bundles
$TM^2|_{M^2\setminus \Sigma}$ and
$\E|_{M^2\setminus \Sigma}$.
Moreover,
$D$ can be identified 
with the Levi-Civita connection 
of the metric $d\sigma^2$
on 
$M\setminus \Sigma$ by $\psi$,
since
$\{\omega_{\alpha}\}_{\alpha\in \Lambda}$
consists of  the connection forms of the Levi-Civita
connection.
Hence
it satisfies \eqref{eq:c} on $M^2\setminus \Sigma$.
Then by the continuity, \eqref{eq:c} holds on $M^2$, 
which proves $(\E, \inner{~}{~}, D)$ 
is a coherent tangent bundle.

 By the definition of $\psi$,
 $d\sigma^2$ is the pull-back metric of $\inner{~}{~}$
 by $\psi$ on $M^2\setminus\Sigma$, and the continuity of $\psi$
 implies that $d\sigma^2$ is the first fundamental form of
$\psi$ on $M^2$.
 If $d\sigma^2$ is co-orientable,
 one can choose the family $\{\lambda_\alpha\}$
 so that $\lambda_\alpha \lambda_\beta$
 takes the same sign as $\det(\J_{\alpha\beta})$,
 which implies that
 the determinant of $g_{\alpha\beta}$ given by
 \eqref{gab} is positive. 
 Hence
 \[
    \mu:=\sgn(\lambda_{\alpha})\hat e_1\wedge \hat e_2
 \]
 does not depend on the choice of the index $\alpha$,
 and gives a non-vanishing section of $M^2$
 into $\E^*\wedge \E^*$.
\end{proof}

As a corollary of Theorem \ref{thm:coherent},
we get the following two Gauss-Bonnet formulas:

\begin{Prop} 
 Let $d\sigma^2$ be a Kossowski metric on
 a compact orientable $2$-manifold $M^2$ without boundary.
 Suppose that $d\sigma^2$ admits at most $A_2$ or
 $A_3$-singularities.
 Then 
 its Gaussian curvature $K$ satisfies
 \begin{equation}\label{eq:GB1}
  2\pi\chi(M^2)
   =\int_{M^2}K\,dA+2\int_{\Sigma}\!\kappa_s\, d\tau,
 \end{equation}
 where $\Sigma$ denotes the singular set of the metric
 $d\sigma^2$, and $\kappa_s$ is the singular curvature
 defined by \eqref{eq:ks}, and
$\tau$ is the arclength parameter of the singular curve.
\end{Prop}
For compact wave fronts in $\R^3$, this formula
was shown in Kossowski \cite{K2}. 
(The singular curvature $\kappa_s$ is not defined in \cite{K2}.
Kossowski treated $\kappa_s\,d\tau$ as a differential form.)
\begin{proof}
 Let $(\E, \langle~,~\rangle, D)$ be the coherent tangent bundle
 induced by the Kossowski metric $d\sigma^2$ as in
 Theorem~\ref{thm:coherent}.
 Then, the singular curvature function $\kappa_s$
 is defined by \cite[(1.7)]{SUY2}, and one can easily check that
 it has the expression \eqref{eq:ks} which can be proved
 by modifying the proof of \cite[Proposition 1.8]{SUY}.
 Moreover, the proof of \cite[Proposition 2.11]{SUY2}
 implies that $\kappa_s\,d\tau$ gives a $C^\infty$-form on
 $1$-dimensional manifold $\Sigma$. 
 Thus \eqref{eq:GB1} holds by applying by the second identity of %
 \cite[Theorem B]{SUY2}.
 We remark that the proof of \cite[Theorem B]{SUY2}
 is given under the assumption that $\E$ is orientable.
 However, taking a double covering of $M^2$ if necessary,
 we may assume that $\E$ is orientable.
 Since the unsigned area element $dA$ is invariant under the
 covering transformation, we get the identity without assuming the
 orientability of $\E$.
\end{proof}

\begin{Prop} 
 Let $d\sigma^2$ be a co-oriented
Kossowski metric on
 a compact oriented $2$-manifold $M^2$ without boundary.
 Suppose that $d\sigma^2$ admits at most $A_2$ and
 $A_3$-singularities.
 Then the following identity holds:
 \begin{equation}\label{eq:Euler}
  \chi^{}_{\E}:=\frac{1}{2\pi}\int_{M^2}K\, d\hat A=
        \chi(M_+)-\chi(M_-)
	+\#S_+-\#S_- ,
 \end{equation}
 where $\chi^{}_{\E}$ is the Euler characteristic 
of the  oriented coherent tangent vector bundle 
 $(\E, \inner{~}{~}, D)$ 
 associated to $d\sigma^2$, 
 $M_+$ {\rm(}resp.\ $M_-${\rm)} is the subset where
 $d\hat A$ is positively {\rm(}resp.\ negatively{\rm)}
 proportional to $dA$,
 and  $\#S_+$ {\rm(}resp.\ $\#S_-${\rm)} is the
 number of positive {\rm(}resp. negative{\rm)}
 $A_3$-points%
\footnote{
 An $A_3$-point $p$ of a given Kossowski metric $d\sigma^2$
 is called {\it positive} (resp. {\it negative})
 if the interior angle of $M_+$ (resp. $M_-$)
 at $p$ is $2\pi$. (In this case,
 the interior angle of $M_-$ (resp. $M_+$)
 at $p$ is zero.)
}.
\end{Prop}
For compact wave fronts in $\R^3$, this formula
was shown in 
Langevin, Levitt and Rosenberg \cite{LLR}. 
\begin{proof}
 The identity
 \eqref{eq:GB1} holds by applying by the first identity of
 \cite[Theorem B]{SUY2}.
\end{proof}

We get the following corollary:

\begin{Cor}
 Let $d\sigma^2$ be a co-orientable Kossowski metric on
 an orientable compact $2$-manifold $M^2$ without boundary.
 Suppose that $d\sigma^2$ admits at most $A_2$ and
 $A_3$-singularities.
 Then the number of $A_3$-points is even.
\end{Cor}

\section{Whitney metrics and cross caps}
\label{sec:W}

Let $U$ be a neighborhood of the origin $o$ in the $uv$-plane
$\R^2$.
A $C^\infty$-map germ $f:(U,o)\to \R^3$ is called a
{\it cross cap\/} if there exists a diffeomorphism germ $\phi$
(resp.\ $\Phi$) of $\R^2$ (of $\R^3$)
such that
$\Phi\circ f\circ \phi(u,v)=(u,uv,v^2)$. 
Whitney \cite{Wh} gave a useful criterion
for this singularities.
If $(0,0)$ is a cross cap singularity of $f$, then
West \cite{W} showed that
there exists a local coordinate system $(u,v)$, and a motion
$T$ in  $\R^3$ such that
\begin{equation}\label{eq:cross}
    T\circ f(u,v)=
        \left(
	   u,
	   uv+\sum_{i=3}^n \frac{b_i}{i!}v^i,
	   \sum_{r=2}^n \sum_{j=0}^r \frac{a_{j\,r-j}}{j!(r-j)!}u^j
	   v^{r-j}
	\right) +O(u,v)^{n+1},
\end{equation}
where $a_{02}\neq 0$ and $O(u,v)^{n+1}$
is a higher order term 
(cf. \cite{W}, see also \cite{FH} and \cite{HHNUY}).
By orientation preserving coordinate changes $(u,v)\mapsto (-u,-v)$
and $(x,y,z)\mapsto (-x,y,-z)$, we may assume that
\begin{equation}\label{eq:a02}
    a_{02}>0,
\end{equation}
where $(x,y,z)$ is the usual Cartesian coordinate system of $\R^3$.
After this normalization \eqref{eq:a02}, one can regard 
all of the coefficients $a_{jk}$ and $b_i$ as
invariants of cross caps.
An oriented local coordinate system $(u,v)$ giving 
such a normal form is called the {\it canonical coordinate system\/} of 
$f$ at the cross cap singularity.
In \cite{HHNUY}, isometric deformations of cross caps
are constructed (cf.\ Figure \ref{fig:deform}),
and it was shown that the coefficients
$a_{03}$, $a_{12}$, $b_3$ 
can be changed by such deformations.

\begin{figure}[htb]
 \begin{center}
        \includegraphics[height=3.0cm]{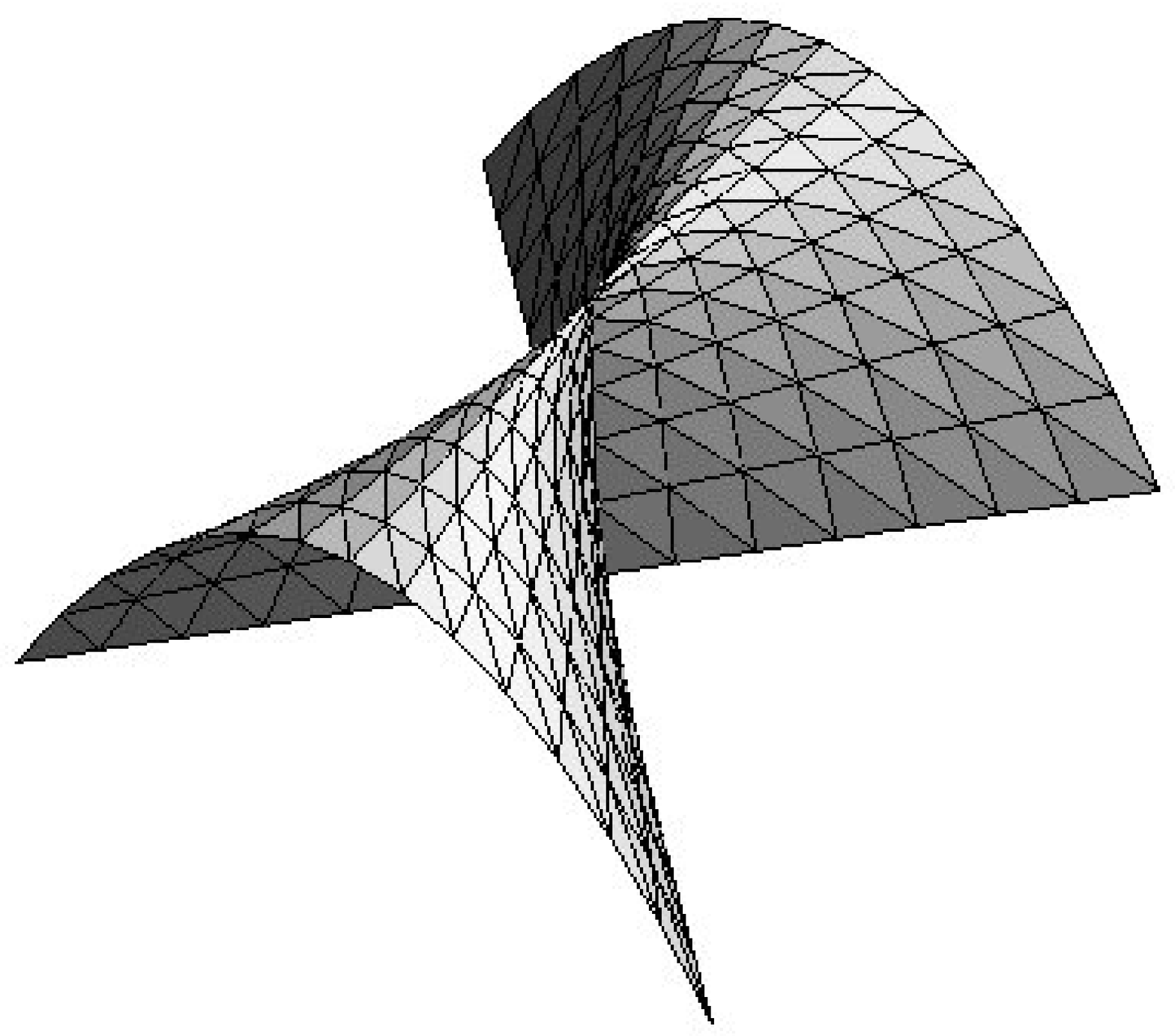}
\quad  \includegraphics[height=3.0cm]{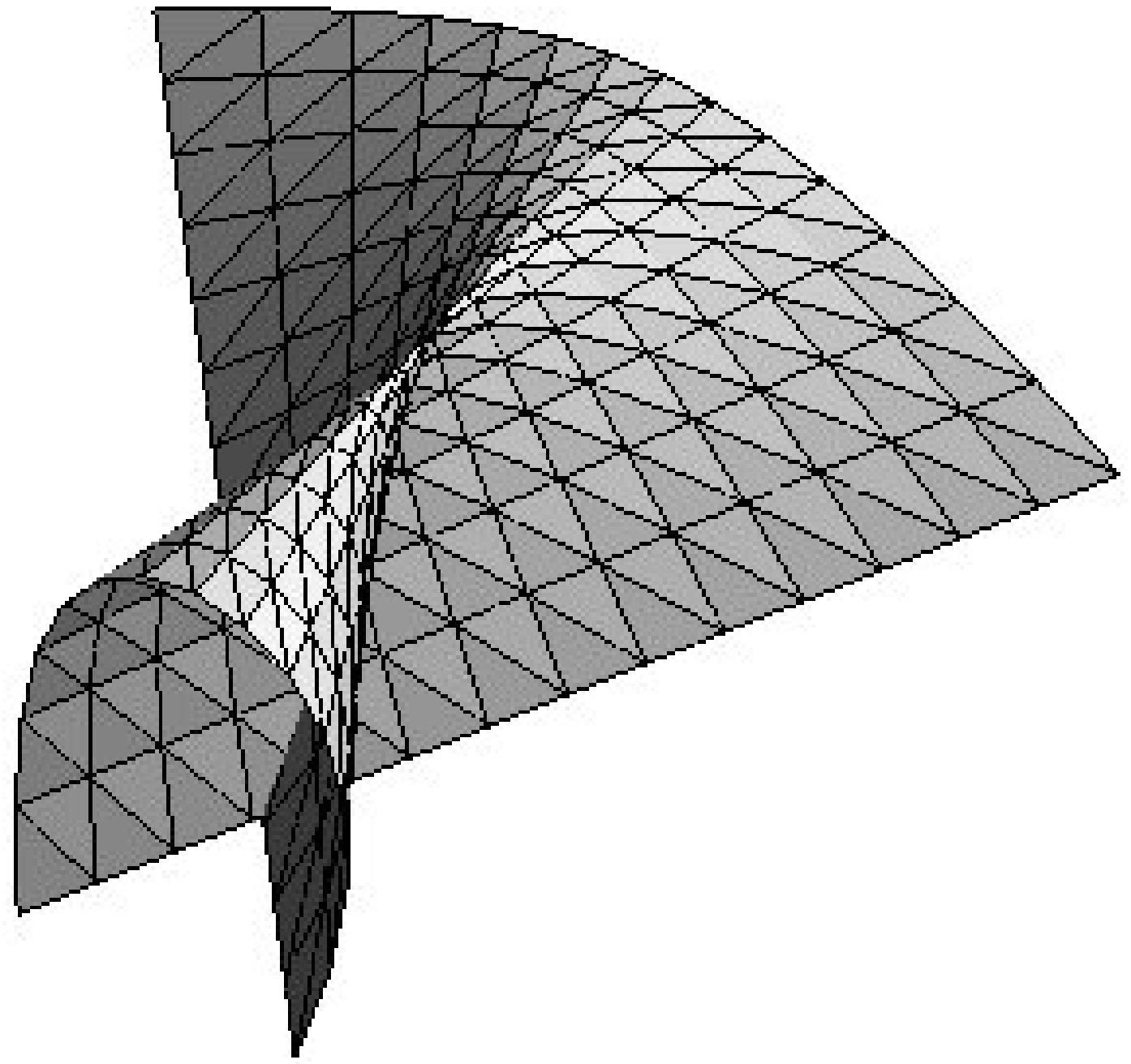}
\quad  \includegraphics[height=3.0cm]{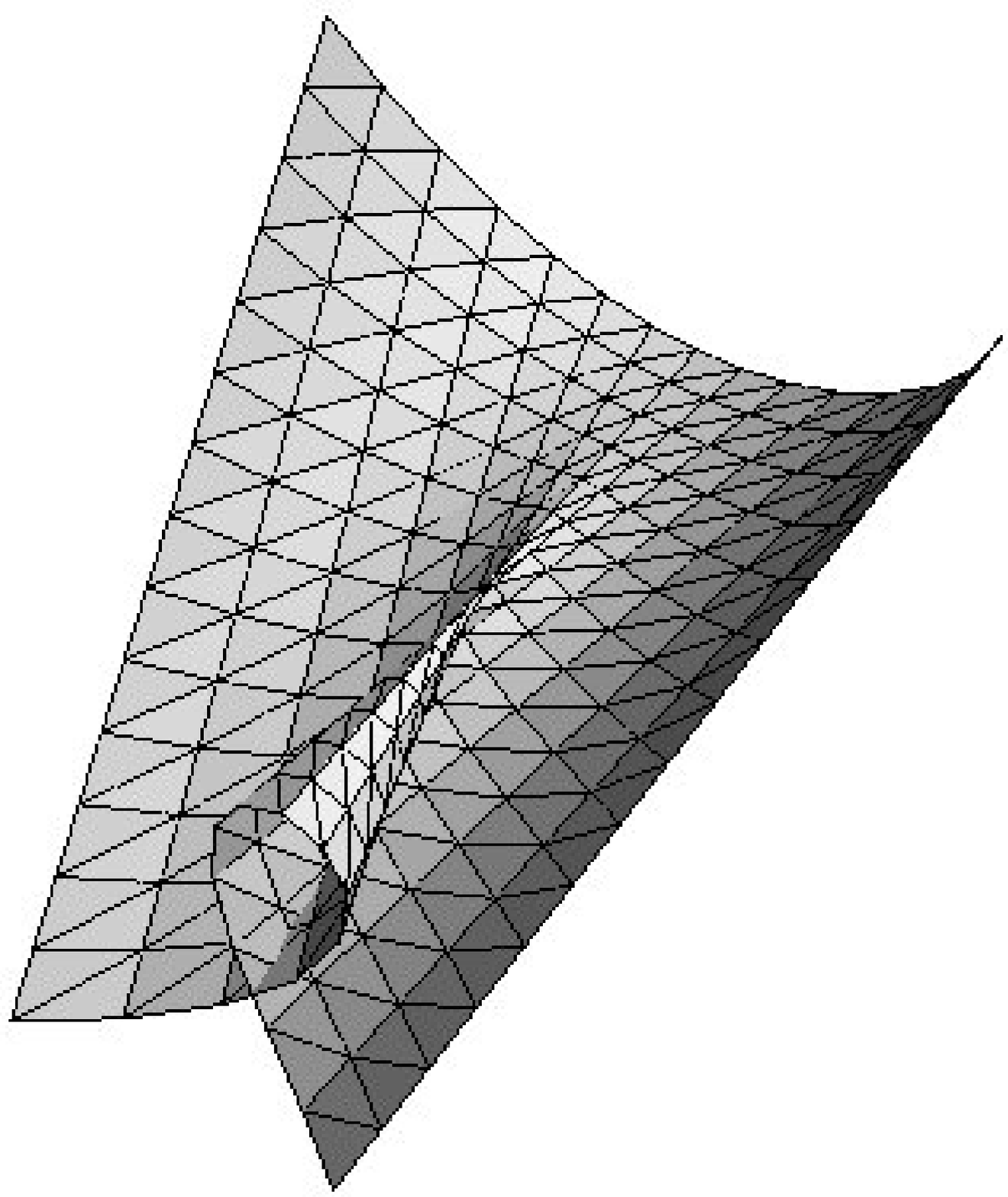}
\caption{An isometric deformation of the standard cross cap.}
\label{fig:deform}
\end{center}
\end{figure}

On the other hand, it was shown 
in \cite{HHNUY} that
$a_{02}$, $a_{20}$, $a_{11}$ 
are all intrinsic invariants.
As a refinement of this fact,
we will now define
a new class of 
semi-definite metrics
called \lq Whitney metrics\rq\,
and will reformulate
$a_{02},\,\,a_{20},\,\, a_{11}$ 
as invariants of isolated singularities
of such metrics, as follows: 
\begin{Def}\label{def:hessian}
 Let $M^2$ be a $2$-manifold, and
 $p$ a singular point of 
 an admissible (positive semi-definite) metric $d\sigma^2$
 on $M^2$ in the sense of Definition~\ref{def:adms}.
 Let $(u,v)$ be a local coordinate system
 centered at $p$ and set
 \[
    \delta:=EG-F^2,
 \]
 where $d\sigma^2=E\,du^2+2Fdu\,dv+G\,dv^2$.
 If the Hessian
 \[
    \Hess_{u,v}(\delta):=
    \det\pmt{\delta_{uu} & \delta_{uv}\\ \delta_{uv} & \delta_{vv}}
 \]
 does not vanish at $p$, then $p$ is called
 an {\it intrinsic cross cap\/} of
 $d\sigma^2$
(cf.\ Corollary~\ref{cor:delta}). 
 Moreover, if $d\sigma^2$ admits only intrinsic cross cap
 singularities, then it is called a {\it Whitney metric}.
\end{Def}

We fix an intrinsic cross cap $p$ of a given Whitney metric
$d\sigma^2$
on $M^2$. 
Let $(u,v)$ be a coordinate system as in Definition \ref{def:hessian}. 
We take  another local coordinate system
$(\xi,\eta)$ centered at $p$.
So it holds that
\begin{equation}\label{eq:J}
 \tilde \delta=J^2 \delta\qquad 
 \left(J:=\det\pmt{u_\xi & u_\eta \\ v_\xi & v_\eta}\right),
\end{equation}
where
$\tilde \delta:=\tilde E\tilde G-\tilde F^2$
and $d\sigma^2=\tilde E\,d\xi^2+2\tilde F\,d\xi\, d\eta+\tilde G\,d\eta^2$.
Since $\delta\ge 0$, 
intrinsic cross caps are
non-degenerate critical points of $\delta$, and in particular
are local
minima of $\delta$.
Hence it holds that
\begin{equation}\label{eq:delta-minima}
 \delta_u(0,0)=\delta_v(0,0)=0,\qquad\text{and}\qquad
 \delta(u,v)>0 \qquad
 \text{on $V\setminus\{p\}$}
\end{equation}
for some neighborhood $V$ of $p$.
Using these,
we have
\begin{equation}\label{eq:detJ}
 \biggl. \Hess_{\xi,\eta}(\tilde \delta)\biggr|_{(0,0)}
    =J(0,0)^2
    \biggl. \Hess_{u,v}(\delta)\biggr|_{(0,0)},
\end{equation}
which implies that the definition of intrinsic cross caps
is independent of the choice of local coordinate systems.

\begin{Exa}
 On the $uv$-plane $\R^2$, we set
 \[
    d\sigma^2:=e^{\epsilon(u,v)}\,du^2+2F(u,v)\,du\,dv+G(u,v)\, dv^2.
 \]
 If the functions $\epsilon$, $F$ and $G$ satisfy
 \begin{align*}
  & \epsilon(0,0)=F(0,0)=G(0,0)=0,\quad
    d \epsilon(0,0)=d F(0,0)=d G(0,0)=0, \\
  & G(u,v)\ge 0, \qquad
  \det\pmt{G_{uu}(0,0) & G_{uv}(0,0)\\
           G_{uv}(0,0) & G_{vv}(0,0)}
  \ne 0,
 \end{align*}
 then it can be checked that $(0,0)$
 is an intrinsic cross cap of $d\sigma^2$.
 Conversely,
 any Whitney metric has such an expression
 at its singular point, which is a 
 consequence of the existence of 
 adapted coordinate systems 
 (cf.\ Definition \ref{def:Wa} and Proposition \ref{prop:Wa}).
\end{Exa}

\begin{Prop}\label{prop:hessian}
 An intrinsic cross cap singular point of $d\sigma^2$ 
 is  an isolated singular point where
 the null space \rm{(}cf.~\eqref{eq:N}{\rm)}
 $\N_p$ is one dimensional.
\end{Prop}
In particular, a Kossowski metric cannot be
a Whitney metric, since singular points of
a Kossowski metric are not isolated (cf.\ Lemma~\ref{lem:rankone}). 
\begin{proof}
 Since non-degenerate critical points of a smooth function
 are isolated, an intrinsic cross cap is an isolated
 singular point.
 Suppose that $\dim \N_p=2$.
 Then
 it holds that
 \[
     E(0,0)=F(0,0)=G(0,0)=0,
 \]
 where $d\sigma^2=E\,du^2+2F\,du\,dv+G\,dv^2$.
 Since $d\sigma^2$ is positive semi-definite,
 $E$, $G\ge 0$ hold, and 
$(0,0)$ is a critical points of the functions
$E,F$. Thus 
we have
 \[
    E_u(0,0)=E_v(0,0)=G_{u}(0,0)=G_{v}(0,0)=0.
 \]
 Then we get
 \[
    \left.\Hess_{u,v}(\delta)\right|_{(0,0)}
       =4\det\pmt{-F_u(0,0)^2 & -F_u(0,0)F_v(0,0)\\
       -F_u(0,0)F_v(0,0) & -F_v(0,0)^2}=0,
 \]
 which contradicts that $p$ is an intrinsic
 cross cap.
\end{proof}

\begin{Prop}\label{prop:delta}
 Let $p$
 be an intrinsic cross cap  
 of a Whitney metric $d\sigma^2$,
 and $(u,v)$ a local coordinate system
 adjusted at $p$ in the sense of Definition~\ref{def:adms-coord}.
 Then the identity
 \[
    \delta_{uu}\delta_{vv}-\delta_{uv}^2=4E\Delta,
    \qquad
    \Delta:=
    \left.
    \det\pmt{E    & F_u      & F_v \\
             F_u  & G_{uu}/2 & G_{uv}/2 \\
             F_v  & G_{uv}/2 & G_{vv}/2}
   \right|_{(0,0)}
 \]
 holds at $(u,v)=(0,0)$, where $\delta:=EG-F^2$ and
 $d\sigma^2=\,Edu^2+2F\,du\,dv+G\,dv^2$.
\end{Prop}
\begin{proof}
 Since $p$ is an admissible singular point of
 $d\sigma^2$ and $(u,v)$ is a local
 coordinate system adjusted at $p$, 
 \[
    F(0,0)=G(0,0)=G_u(0,0)=G_v(0,0)=0
 \]
 hold (cf.\ Proposition \ref{prop:property}).
 Differentiating
 $\delta:=EG-F^2$ twice, we have
 \[
   \delta_{uu}=E G_{uu}-2 F_u^2,\quad
   \delta_{uv}=E G_{uv}-2 F_uF_v,\quad
   \delta_{vv}=E G_{vv}-2 F_v^2
 \]
 at $p=(0,0)$.
 Then one can get the identity by a straightforward calculation.
\end{proof}

As a corollary, we can show the following assertion,
which is a reason why we call $p$
an \lq intrinsic cross cap\rq. 
\begin{Cor}\label{cor:delta}
 Let $f:M^2\to \R^3$  be a $C^\infty$-map
 and $p$ a cross cap singularity of $f$.
 Then the first fundamental form $d\sigma^2$ of $f$
 is an admissible metric, and $p$ is an intrinsic
 cross cap of $d\sigma^2$.
\end{Cor}
\begin{proof}
 By Proposition \ref{prop:exa-adms},
 $p$ is an admissible singular point of
 the metric $d\sigma^2$.
 We take a local coordinate system $(u,v)$ centered at $p$
 such that $f_v(0,0)=0$, in particular,
 $(u,v)$ is a local coordinate system adjusted at $p$.
By a well-known criterion of cross caps by Whitney \cite{Wh},
the three vectors
 \[
    f_u(0,0),\quad f_{uv}(0,0),\quad f_{vv}(0,0)
 \]
 must be linearly independent at $f(p)$.
 In particular,
 \begin{equation}\label{eq;det0}
  \pmt{f_u \\ f_{uv} \\ f_{vv}}
   (f_u,f_{uv},f_{vv})
   =
   \pmt{E    & F_u      & F_v \\
  F_u  & G_{uu}/2 & G_{uv}/2 \\
             F_v  & G_{uv}/2 & G_{vv}/2}
 \end{equation}
 is a regular matrix,
 where
 \[
    E=f_u\cdot f_u,\quad F=f_u\cdot f_v,\quad G=f_v\cdot f_v.
 \]
 Taking the determinant of
 \eqref{eq;det0},
 the conclusion follows from
 Proposition \ref{prop:delta}.
\end{proof}

The following assertion gives a characterization of
the coefficient $a_{02}$ of cross caps
in terms of Whitney metrics.

\begin{Prop}\label{prop:Wchr}
 Let $p\in M^2$ be an intrinsic cross cap of
 a given Whitney metric $d\sigma^2$.
 Let $(u,v)$ be a local coordinate system adjusted
 at $p$ {\rm (}cf.\ Definition~\ref{def:adms-coord}{\rm)}
 and set
 \[
     \alpha:=\frac{E(0,0) G_{vv}(0,0)}2-F_v(0,0)^2,
 \]
 where $d\sigma^2=E\,du^2+2F\,du\,dv+G\,dv^2$.
 Then 
 \begin{equation}\label{eq:a02def}
  \alpha_{02}:=\frac{\sqrt{E(0,0)}\alpha^{3/2}}{\Delta}
 \end{equation}
 is a positive value, which 
 does not depend on the choice of
 local coordinate systems adjusted at $p$
 {\rm (}cf.\ Definition~\ref{def:adms-coord}{\rm)}.
 Moreover, $\alpha_{02}$ coincides
 with the coefficient\footnote{
There is a typographical error in \cite[Page 779]{HHNUY}.
The right-hand side of the expression of 
$a_{02}$ should be 
divided not by $2$ but by $2^{3/2}$.
} 
$a_{02}$ in \eqref{eq:cross} if $d\sigma^2$
 is induced by the cross cap in $\R^3$.
\end{Prop}
\begin{proof}
It can be easily checked that
$\alpha=\delta_{vv}(0,0)/2$
where $\delta=EG-F^2$.
Since $\Hess_{u,v}(\delta)\ne 0$,
we have that $\alpha>0$.
On the other hand,
by Proposition \ref{prop:delta}, it holds that
\[
   4E(0,0)\Delta=\delta_{uu}(0,0)\delta_{vv}(0,0)-\delta_{uv}(0,0)^2>0,
\]
and we get the inequality $\alpha_{02}>0$. 

Let $(\xi,\eta)$ be another local coordinate system
adjusted at $p$.
We set
$\tilde \delta=\tilde E\tilde G-\tilde F^2$.
By \eqref{eq:admissible} and
\eqref{eq:u2}, it holds that 
\begin{equation}\label{eq:tE}
\tilde E(0,0)=E(0,0) u_\xi(0,0)^2.
\end{equation}
By \eqref{eq:J} and
\eqref{eq:u2}, we have
\begin{equation}\label{eq:hth}
\tilde \delta(0,0)=\biggl(u_\xi(0,0) v_\eta(0,0)\biggr)^2 \delta(0,0).
\end{equation}
By \eqref{eq:u2},  
$\tilde \delta_{\eta\eta}(0,0)=\tilde \delta_{vv}(0,0) v_\eta(0,0)^2$
holds. 
Moreover, since $\delta=\delta_u=\delta_v=0$ holds at $(0,0)$, we have that
\[
\tilde \delta_{\eta\eta}(0,0)=\delta_{vv}(0,0)(u_\xi(0,0))^4(v_\eta(0,0))^6.
\]
Then \eqref{eq:tE} yields 
\begin{equation}\label{eq:001}
(\tilde E(0,0) \tilde \delta_{\eta\eta}(0,0))^{3/2}
=(u_\xi v_\eta)^6\biggl(E(0,0) \delta_{vv}(0,0)\biggr)^{3/2}.
\end{equation}
On the other hand, \ref{eq:detJ}, we have
\[
\Hess_{\xi,\eta}(\tilde \delta)(0,0)
=(u_\xi v_\eta)^6\op{Hess}_{u,v}(\delta)(0,0).
\]
By this equality and \eqref{eq:001}, we get
\begin{equation}\label{eq:id0}
\frac{(\tilde E(0,0) 
\tilde \delta_{\eta\eta}(0,0))^{3/2}}{\Hess_{\xi,\eta}(\tilde \delta)(0,0)}=
\frac{(E(0,0) \delta_{vv}(0,0))^{3/2}}{\Hess_{u,v}(\delta)(0,0)},
\end{equation}
which implies the coordinate invariance of $\alpha_{02}$.
Then the last assertion follows from
\cite[Corollary 8]{HHNUY}.
\end{proof}

We next give formulations for $a_{20}$ and $a_{11}$
in terms of Whitney metrics.
\begin{Def}\label{def:Wa}
 Let $p$ be a singular point of a Whitney metric.
 If a local coordinate system $(u,v)$ adjusted at $p$
 satisfies
 \[
    E(0,0)=1,\qquad d E(0,0)=d F(0,0)=d G(0,0)=0,
 \]
 then it is called an {\em adapted coordinate system}
 at $p$, where $d\sigma^2=E\,du^2+2F\,du\,dv+G\,dv^2$.
\end{Def}
The following assertion can be easily proved.
\begin{Prop}\label{prop:Wa}
 There exists an adapted coordinate system at a given intrinsic
 cross cap.
 Moreover, if $(u,v)$ and $(\xi, \eta)$ are two
 adapted coordinate systems at 
 $p$, then 
 \begin{equation}\label{eq:A}
   u_\xi(0,0)=\pm 1,\qquad
   u_\eta(0,0)=u_{\xi\xi}(0,0)=u_{\xi\eta}(0,0)=u_{\eta\eta}(0,0)=0
 \end{equation}
 hold.
 Conversely, under the assumption that
 $(u,v)$ is adapted,
 a new coordinate system
 $(\xi,\eta)$ adjusted at $p$
 is also an adapted coordinate at $p$
 if it satisfies
 \eqref{eq:A}.
\end{Prop}
\begin{proof}
 Let $(u,v)$ be an adjusted coordinate system centered at $p$,
 and set
 \[
   u(\xi,\eta)=c_1 \xi+c_{11} \xi^2+c_{12}\xi\eta+c_{22}\eta^2, \qquad
   v(\xi,\eta)=\eta
 \]
 for constants $c_1$, $c_{11}$, $c_{12}$ and $c_{22}$.
 Then one can choose these constants such that $(\xi,\eta)$
 is adapted,
 using \eqref{eq:ee}, 
 \eqref{eq:ff} and \eqref{eq:gg}.
 The last assertion follows immediately.
\end{proof}

\begin{Def}\label{def:West}
 An adapted coordinate system $(u,v)$ at an intrinsic cross cap $p\in M^2$
 is called  
a {\it West type coordinate system of the second  order}
 if there exist two  real numbers
 $\alpha_{20}$ and $\alpha_{11}$ such that
 \begin{align*}
  E&=1+(\alpha_{20})^2u^2+2 \alpha_{11}\alpha_{20}uv+(1+(\alpha_{11})^2)v^2
  +O(u,v)^3,\\
  F&=\alpha_{11}\alpha_{20}u^2+
  (1+(\alpha_{11})^2+\alpha_{02}\alpha_{20})uv+\alpha_{02}\alpha_{11}v^2
  +O(u,v)^3,\\
  G&=(1+(\alpha_{11})^2)u^2+2 \alpha_{02}\alpha_{11}uv+(\alpha_{02})^2 v^2
  +O(u,v)^3,
 \end{align*}
 where $d\sigma^2=E\,du^2+2F\,du\,dv+G\,dv^2$ and
 $\alpha_{02}$ is given by \eqref{eq:a02def}.
\end{Def}

The following assertion can be proved easily, and is the reason why
we call $(u,v)$ is of West type:

\begin{Prop} 
 A canonical coordinate system at a cross cap singular point
 is a West type coordinate system of the second order.
 Moreover, 
 $\alpha_{20}$ and $\alpha_{11}$ coincide with
 the corresponding coefficients 
$a_{20}$ and $a_{11}$
 in \eqref{eq:cross}.
\end{Prop}

The following assertion holds:

\begin{Thm}\label{thm:west}
There exists a West type coordinate system 
of the second order  at each
 intrinsic cross cap.
\end{Thm}
To prove the theorem, we prepare several definitions and lemmas:
\begin{Def}
 An adapted coordinate system $(u,v)$ at an
 intrinsic cross cap $p$ is said to be
 {\it adjusted in the first-level} 
 if $G_{vv}(0,0)=2(\alpha_{02})^2$,
 where $G=d\sigma^2(\partial_v,\partial_v)$.
\end{Def}

By definition,
a West type coordinate system of the second order
is adjusted in the first-level.

\begin{Lemma}\label{lem:c1}
 There exists an adapted coordinate system 
$(u,v)$ with first-level adjustment at an
 intrinsic cross cap $p$.
 Moreover, if an adapted coordinate system $(\xi,\eta)$
 at $p$ satisfies
 \begin{equation}
  \label{eq:A2}
   v_{\eta}(0,0)=\pm 1,
 \end{equation}
 then $(\xi,\eta)$ is also an
 adapted coordinate system adjusted in the first-level. 
\end{Lemma}
\begin{proof}
 Let $(u,v)$ be an adapted coordinate system at $p$.
 Then the new coordinate system $(\xi,\eta)$
 given by
 \[
   u=\xi,\qquad v=\frac{%
   \sqrt[4]{2}\sqrt{\alpha_{02}}}{\sqrt[4]{G_{vv}(0,0)}}\eta
 \]
 has the desired property.
\end{proof}
\begin{Rmk}\label{rmk:alpha}
 Let $(u,v)$ be an adapted coordinate system
 adjusted in the first-level 
 at  an intrinsic cross cap $p$.
 Since $E(0,0)=1$ and $F_v(0,0)=0$ by the definition of 
 adaptedness (cf.\ Definition~\ref{def:Wa}),
 $\alpha$ in Proposition~\ref{prop:Wchr} satisfies
 $\alpha=\frac{G_{vv}(0,0)}2$.
 By $G_{vv}(0,0)=2(a_{02})^2$, we have
 $\alpha=(\alpha_{02})^2$.
 So we have
 \[
   \Delta=\frac{\alpha^{3/2}}{\alpha_{02}}
=(\alpha_{02})^2=\frac{G_{vv}(0,0)}2.
 \]
\end{Rmk}

\begin{Def}\label{def:c2}
 An adapted coordinate system $(u,v)$ 
 adjusted in the first-level 
 is said to be {\it adjusted in the second-level} 
 if
 \begin{equation}\label{eq:second}
    \det\pmt{F_{uu}-E_{uv}/2 & G_{uv} \\
             F_{uv}-E_{vv}/2 & G_{vv}}=0
 \end{equation}
 holds at $(0,0)$, where
 $d\sigma^2=E\,du^2+2F\,du\,dv+G\,dv^2$.
\end{Def}

By Definition \ref{def:West}, it can be easily checked that
a West type coordinate system of the second order
is adjusted in the second-level.

\begin{Lemma}\label{prop:c2}
 There exists an adapted coordinate system at an
 intrinsic cross cap $p$ with second-level adjustment.
 Moreover, 
under the assumption that $(u,v)$ 
is an adapted coordinate system with second-level
adjustment,
 an adapted coordinate system $(\xi,\eta)$
 at $p$ is adjusted in the second level if and only if
 \begin{equation}\label{eq:A3}
  v_{\eta}(0,0)=\pm 1,\qquad v_{\xi}(0,0)=0.
 \end{equation}
\end{Lemma}

\begin{proof}
 Let $(u,v)$ and $(\xi,\eta)$
 be two adapted coordinate systems
 with first-level adjustment. 
For the sake of simplicity, we consider the case of $u_{\xi}(0,0)=v_{\eta}(0,0)=1$ in \eqref{eq:A} and \eqref{eq:A2}.
 Then by \eqref{eq:u2},
 it holds that
 \[
   \det\begin{pmatrix}
	\tilde F_{\xi\xi}-\tilde E_{\xi\eta}/2 &
	\tilde G_{\xi\eta}
	\\
	\tilde F_{\xi\eta}-\tilde E_{\eta\eta}/2 &
	\tilde G_{\eta\eta}
       \end{pmatrix}
   =
    \det\pmt{F_{uu}-E_{uv}/2 & G_{uv} \\
             F_{uv}-E_{vv}/2 & G_{vv}}
      + v_{\xi}
    \det\pmt{G_{uu} & G_{uv} \\ G_{vu} &G_{vv} }
 \]
 at the origin,
 where $d\sigma^2=\tilde E\,d\xi^2+2\tilde F\,d\xi\,d\eta+\tilde
 G\,d\eta^2$.
 Since $(u,v)$ is adapted, $E=1$ and $F_u=F_v=0$ at the origin.
 Then by Proposition~\ref{prop:delta},
 \begin{equation}\label{eq:Delta}
   \Hess_{u,v}\delta(0,0)
   = 4\Delta=4\det\pmt{G_{uu}(0,0)& G_{uv}(0,0)\\
                       G_{vu}(0,0)& G_{vv}(0,0)}\neq 0.
 \end{equation}
 Thus, if we set
 \[
    u=\xi,\qquad v=\eta+c \xi,
 \]
 for a suitable constant $c$, such that
 \eqref{eq:second} holds at $(0,0)$,
then $(u,v)$ is a desired  adapted coordinate system
 with second-level adjustment.
\end{proof}

Now, we take an adapted coordinate system $(u,v)$ 
adjusted in the second-level. 
By Remark~\ref{rmk:alpha} and \eqref{eq:Delta},
\begin{align}\label{eq:G1}
 & G_{vv}(0,0)=2(\alpha_{02})^2,\\
 \label{eq:G2}
 & 4(\alpha_{02})^2=G_{uu}(0,0)G_{vv}(0,0)-G_{uv}^2(0,0)
\end{align}
hold. We now define a constant $\alpha_{11}$ so that
\begin{equation}\label{eq:a1-1}
   G_{uv}(0,0)=2 \alpha_{11}\alpha_{02}.
\end{equation}
Then \eqref{eq:G1} and \eqref{eq:G2} yield that
\begin{equation}\label{eq:a1-2}
   G_{uu}(0,0)=2(1+(\alpha_{11})^2).
\end{equation}
Next,  we define a constant $\alpha_{20}$ by
\begin{equation}\label{eq:FE1}
   F_{uv}-\frac{E_{vv}}2=\alpha_{20}\alpha_{02}.
\end{equation}
Then it holds that
\begin{align*}
  0&=
  \det\pmt{F_{uu}-E_{uv}/2 & G_{uv} \\
             F_{uv}-E_{vv}/2 & G_{vv}}
  \\&=
  \det\pmt{F_{uu}-E_{uv}/2 & 2 \alpha_{11}\alpha_{02} \\
             \alpha_{20}\alpha_{02} & 2(\alpha_{02})^2}
   =2(\alpha_{02})^2
   \det\pmt{F_{uu}-E_{uv}/2 & \alpha_{11} \\
             \alpha_{20} & 1},
\end{align*}
and we get
\begin{equation}\label{eq:FE2}
   F_{uu}-\frac{E_{uv}}2=\alpha_{11}\alpha_{20}.
\end{equation}

\begin{Thm}
The two constants
 $\alpha_{20}$ and $|\alpha_{11}|$ 
 at each intrinsic cross cap singularity
 do not depend on the choice
 of an  adapted coordinate system $(u,v)$
 with second-level adjustment.
Moreover, if $M^2$ is oriented, then
$\alpha_{11}$ is 
independent of the choice
of oriented adapted coordinate systems 
with second-level adjustment.
\end{Thm}

\begin{proof}
The coordinate independence of
$\alpha_{20}$ and $|\alpha_{11}|$ are proved by using
\eqref{eq:ee}, \eqref{eq:ff}, \eqref{eq:gg}, 
\eqref{eq:A}, \eqref{eq:A3},
\eqref{eq:a1-1}
and \eqref{eq:FE1}.
Let $(\xi,\eta)$ be another
adapted coordinate system $(u,v)$
with second-level adjustment.
Then it holds that
$$
\tilde G_{\xi\eta}(0,0)
=u_\xi(0,0)v_\eta(0,0)G_{uv}(0,0).
$$
If $M^2$ is oriented, then
an orientation preserving coordinate change
between adapted coordinate systems 
with second-level adjustment
should satisfy
$$
u_\xi(0,0)v_\eta(0,0)=1.
$$
Thus $G_{uv}(0,0)$ is independent of
such a coordinate change, and 
the equality \eqref{eq:a1-1}
implies the desired coordinate invariance 
of $\alpha_{11}$.
\end{proof}

\begin{Cor}
Let $(u,v)$ be  an oriented
adapted coordinate system with second-level adjustment.
If we set
\[
   u=r\cos \theta,\qquad v=r\sin \theta,
\]
then the Gaussian curvature $K$ of $d\sigma^2$
satisfies
\begin{equation}\label{eq:Gauss_in}
\lim_{r \to 0}r^2K
=\frac{\alpha_{02}(\alpha_{20}\cos^2\theta -\alpha_{02}\sin^2 \theta)}
{\left(\cos^2\theta+(\alpha_{11}\cos \theta+\alpha_{02}\sin \theta)^2\right)^2}.\end{equation}
\end{Cor}

\begin{proof}
By \eqref{eq:G1}, \eqref{eq:a1-1}, \eqref{eq:a1-2},
\eqref{eq:FE1} and \eqref{eq:FE2} 
we can write
 \begin{align*}
  E&=1+e_{20} u^2+2 e_{11}uv+e_{02}v^2
  +O(u,v)^3,\\
  F&=\frac12(\alpha_{11}\alpha_{20}+e_{11})u^2+
  (e_{02}+\alpha_{02}\alpha_{20})uv+f_{02}v^2
  +O(u,v)^3,\\
  G&=(1+(\alpha_{11})^2)u^2+2 \alpha_{02}\alpha_{11}uv+(\alpha_{02})^2 v^2
  +O(u,v)^3,
 \end{align*}
where $e_{20},e_{11},e_{02}$ and $f_{02}$ are
constants.
Substitute them into
the formula 
\eqref{eq:egregium}.
Then one can directly check that
the limit does not depend on the
constants
$e_{20},e_{11},e_{02}$ and $f_{02}$
and can get \eqref{eq:Gauss_in}.
\end{proof}

\begin{Rmk}
The formula \eqref{eq:Gauss_in} coincides with \cite[(20)]{HHNUY},
that is, the primary divergent term of 
the Gaussian curvature of
an intrinsic cross cap coincides with that of 
a cross cap in $\R^3$,
since the canonical coordinate system of 
a cross cap is
an adapted coordinate system  with 
second-level adjustment.
\end{Rmk}

\begin{proof}[Proof of Theorem \ref{thm:west}]
 Let $(u,v)$ 
 be an  adapted coordinate system adjusted in the second-level. 
 We define a new local  coordinate system $(\xi,\eta)$
 by
 \[
  u(\xi,\eta)=\xi+c_{30}
   \xi^3+c_{21}\xi^2\eta+c_{12}\xi\eta^2+c_{03}\eta^3, 
   \qquad
   v(\xi,\eta)=\eta
 \]
 and can adjust the above four coefficients.
 In fact, $c_{30}$ is adjusted to get 
$E_{uu}=(\alpha_{20})^2$,
 $c_{21}$ is adjusted for
 $F_{uu}=2\alpha_{11}\alpha_{20}$,
 $c_{21}$ is for
 $F_{uv}=1+(\alpha_{11})^2+\alpha_{02}\alpha_{11}$,
 and
 $c_{03}$ is for
 $F_{vv}=2\alpha_{02}\alpha_{11}$.
 Then $E_{uu}$ and $E_{uv}$ are determined by
 \eqref{eq:FE1} and \eqref{eq:FE2}.
\end{proof}

By the existence of a West type coordinate system,
we get the following assertion:

\begin{Cor}
The two invariants $a_{20}$ and $a_{11}$
of cross caps in $\R^3$ can be extended
as corresponding invariants 
$\alpha_{20}$ and $\alpha_{11}$
for
intrinsic cross caps, respectively.
\end{Cor}

\begin{Rmk}
Let $(0,0)$ be an intrinsic cross cap
singularity of the metric
$d\sigma^2=Edu^2+2Fdudv+Gdv^2$ on
an oriented coordinate neighborhood $(U;u,v)$.
We set
$$
dA:=\sqrt{EG-F^2}du\wedge dv.
$$
By \eqref{eq:Gauss_in}, $KdA$ 
gives a
smooth 2-form with respect to the
polar coordinate system $(r,\theta)$,
where $u=r \cos \theta$ and
$v=r \sin \theta$.
Thus, the integral
$\int_U KdA$ is well-defined.
\end{Rmk}

\section{Gauss-Bonnet formula for Whitney metrics}
\label{sec:WGB}
At the end of this paper, we shall prove the following
Gauss-Bonnet type formula for Whitney metrics:

\begin{Thm}
 Let $M^2$ be a compact oriented manifold without boundary,
 and $d\sigma^2$ a Whitney metric in $M^2$.
 Then its Gaussian curvature $K$ satisfies
 \[
    \frac{1}{2\pi}\int_{M^2} K dA=\chi(M^2),
 \]
 that is, there is no defect at intrinsic cross cap
 singularities for the Gauss-Bonnet formula.
\end{Thm}

For compact surfaces which admit only cross cap
singularities
in $\R^3$, the corresponding  formula
was shown in Kuiper \cite{K}. 

\begin{proof}
 We fix a singular point $p$ of $d\sigma^2$,
 and take an oriented West type coordinate system of the second order 
 $(U;u,v)$ at $p$.
 Using the formula \eqref{eq:Gauss_in},
 one can see that
 $K\,dA$ is a smooth $2$-form on $U$.
 We set
 \[
 D(r):=
 \{(u,v)\in U\,;\, u^2+v^2\le r^2\}
 \]
 and 
 \[
 D_+(r,\epsilon):=D(r)\cap\{v\ge\epsilon\},\qquad
 D_-(r,\epsilon):=D(r)\cap\{v\le -\epsilon\},
 \]
 where $\epsilon$ is a sufficiently small positive constant.
 Then we have that
 \begin{equation}\label{eq:F1}
  \int_{D(r)}K\,dA
   =
   \lim_{\epsilon\to 0} \int_{D_+(r,\epsilon)}K\,dA
   +\lim_{\epsilon\to 0} \int_{D_-(r,\epsilon)}K\,dA.
 \end{equation}
 Moreover, the classical Gauss-Bonnet formula yields that
 \begin{equation}
 \label{eq:F2}
  \int_{D_\pm(r,\epsilon)}K\,dA=
  \int_{\partial D_\pm(r,\epsilon)} \kappa_g\, ds+\pi,
 \end{equation}
 and so
 \[
 \int_{D(r)}KdA
 =
 2\pi+
 \lim_{\epsilon\to 0} 
 \int_{\partial D_+(r,\epsilon)} \kappa_g\, ds
 +\lim_{\epsilon\to 0} 
 \int_{\partial D_-(r,\epsilon)} \kappa_g\, ds
 \]
 holds.
 We set
 \[
 L_+(\epsilon):=\partial D_+(r,\epsilon)\setminus \partial D(r),\qquad
 L_-(\epsilon):=\partial D_-(r,\epsilon)\setminus \partial D(r).
 \]
 Then they are line segments parallel to the $u$-axis.
 Moreover, we may assume that
 the orientation of $L_+(\epsilon)$ is the same as 
 that of the $u$-axis,
 and the orientation of $L_-(\epsilon)$ is opposite
 that of $L_+(\epsilon)$. 
 These orientations are compatible with
 respect to the anti-clockwise orientations
 of $\partial D_\pm(r)$.
 We now compute the geodesic curvature of $L_+(\epsilon)$
 for each fixed sufficiently small $\epsilon(>0)$:
 \[
 \kappa_g\, ds=
 \frac{\langle \nabla_u \dot \gamma^{\epsilon}(u),
\vect{n}(u,\epsilon)\rangle du}{%
  |\dot \gamma^{\epsilon}(u)|^2}
 =\frac{\Gamma(\partial_u,\partial_u, \vect{n})du}{|\dot \gamma^{\epsilon}(u)|^2},
 \]
where $\gamma^\epsilon(u):=(u,\epsilon)$ ($|u|<\sqrt{r^2-\epsilon^2}$)
and
$\vect{n}(u,\epsilon)$ is
 the co-normal vector field along $\gamma^\epsilon$.
 If we set $d\sigma^2=E\,du^2+2F\,du\,dv+G\,dv^2$, we have 
 \[
  \vect{n}(u,\epsilon)=\frac{-F\partial_u+E\partial_v}{\sqrt{E(EG-F^2)}}
 \]
 and by a straightforward calculation, we have
 \[
    \inner{\nabla_u \dot \gamma^\epsilon(u)}{\vect{n}(u,\epsilon)}
     =
      \frac{E(2F_u-E_v)-E_uF}{2\sqrt{E(EG-F^2)}}.
 \]
 Since $|\dot \gamma^{\epsilon}(u)|^2=E(u,\epsilon)$, we get
 \[
    \kappa_g\, ds=\frac{(E(2F_u-E_v)-E_uF)du}{2\sqrt{E^3(EG-F^2)}}.
 \]
 By setting,
 $u=r \cos \theta$,
 $v=r \sin \theta$,
 it holds that
 \[
   \kappa_g\, ds
=\frac{du}{E^{3/2}}\frac{\alpha_{20}(\alpha_{11}\cos\theta+\alpha_{02}\sin\theta)
+r O_1(r,\theta)}{\sqrt{(1+(\alpha_{11})^2)\cos^2\theta+2 \alpha_{02}\alpha_{11}\cos\theta \sin \theta+
(\alpha_{02})^2\sin^2 \theta+r O_2(r,\theta)}},
 \]
 where  $O_i(r,\theta)$ ($i=1,2$) are 
 $C^\infty$-functions of $r,\theta$. 
 Since the light-hand side is bounded,
 we can show that
 \begin{equation}\label{eq:lim1}
  \lim_{\epsilon\to 0}\int_{L_+(\epsilon)}\kappa_g\, ds
   =\int_{L_+(0)}\kappa_g\, ds.
 \end{equation}
 Since $\lim_{\epsilon\to 0}\kappa_g\, ds$ 
is not continuous at $u=0$ as a $1$-form on the
$u$-axis,
 the integrals should be taken to be Lebesgue integrals.
 Similarly, we have
 \begin{equation}\label{eq:lim2}
  \lim_{\epsilon\to 0}\int_{L_-(\epsilon)}\kappa_g\, ds
   =\int_{L_-(0)}\kappa_g ds=-\int_{L_+(0)}\kappa_g\, ds.
 \end{equation}
 Thus, we get
 \begin{equation}\label{eq:lim0}
  \int_{D(r)}K\,dA
   =
   2\pi+
   \int_{\partial D} \kappa_g\, ds.
 \end{equation}
 The relations \eqref{eq:lim0},
 \eqref{eq:F1} and  \eqref{eq:F2}
 imply the assertion.
\end{proof}

It is well-known that even numbers of cross caps
appear in closed surfaces in $\R^3$ which admit only
cross cap singularities.
In the previous section, we have shown that
the number of $A_3$-points is even under the assumption
that the Kossowski metric is co-orientable.
However, the above Gauss-Bonnet formula does not
give any such restriction of the number of intrinsic cross caps.
In fact, one can construct a Whitney metric
on a torus having only one intrinsic cross cap
as follows:
The $C^\infty$-map $f(u,v)=(u,uv,v^2)$ 
has a cross cap singularity at the origin,
and its first fundamental form is given by
\[
ds^2:=\left(1+v^2\right)\,du^2 +2\,uv\, du\,dv+\left(u^2+4v^2\right)\,dv^2.
\]
Let $\rho:\R\to [0,1]$ be a $C^\infty$-function
such that $\rho(t)=1$ for $|t|\le 1/4$ and
$\rho(t)=0$ for $|t|\ge 3/4$.
We set
\[
  d\sigma^2:=\rho(2r) ds^2+(1-\rho(2r))(du^2+dv^2)
    \qquad \left(r:=\sqrt{u^2+v^2}\right).
\]
Then $d\sigma^2$ 
has a singular point only on $(0,0)$, and 
is a Whitney metric having 
an intrinsic cross cap at $(0,0)$,
which is defined on the square-shaped closed
domain 
$\overline D:=\{(u,v)\in \R^2\,;\, -1\le u,v\le 1\}$.
Identifying each of two pairs of the parallel edges of
the boundary of $\bar D$, the 
metric $d\sigma^2$ can be considered as a
Whitney metric on the square torus having
only one cross cap singularity.

\appendix
\section{Wave fronts, cuspidal edges
and swallowtails}

Let $M^2$ be a $2$-manifold and 
$f:M^2\to \R^3$ a $C^\infty$-map.
A point $p\in M^2$ is called {\it regular} if
$f$ is an immersion on a sufficiently small neighborhood
of $p$, and is called {\it singular} if it is not regular. 
A $C^\infty$-map $f:M^2\to \R^3$ is called a 
{\it frontal}
 if for each $p\in M^2$ there exists a
unit normal vector field $\nu$ along $f$
defined on
a neighborhood $U_p$ of $p$.
By parallel displacements in $\R^3$, 
$\nu$ can be
considered as a map $\nu:U_p\to S^2$.
In this case, $\nu$ is called the {\it Gauss map}
of the frontal $f$.
Moreover, if the map
$$
L:=(f,\nu):U_p\to \R^3\times S^2
$$
gives an immersion for each $p\in M^2$, 
$f$ is called a {\it front} or a {\it wave front}.
Using the canonical inner product,
we identify the unit tangent bundle
$
T_1\R^3=\R^3\times S^2
$
with the unit cotangent bundle $T_1^*\R^3$, which has
the canonical contact structure.
When $f$ is a front, $L$ gives a Legendrian immersion.

Let $f:M^2\to \R^3$ be a frontal.
Then $f$ is called {\it co-orientable}
if there exists a smooth unit normal vector
field $\nu$ globally defined on $M^2$. 
We fix a singular point $p\in M^2$ of $f$
and take a local coordinate system $(U_p;u,v)$
of $p$. The function on $U_p$ defined by
$$
\lambda:=\det(f_u,f_v,\nu)
$$
is called a {\it signed area density function}.
The set $\{(u,v)\in U_p\,;\, \lambda(u,v)=0\}$
coincides with the singular set of $f$
on $U_p$. 
A singular point $p$ is called {\it non-degenerate} if
the exterior derivative $d\lambda$ of $\lambda$
does not vanish at $p$.
(This definition does not depend on the choice of
local coordinate systems at $p$.)
On a neighborhood of a non-degenerate singular point,
the singular set consists of a regular curve
$\gamma(t)$, called the 
{\it singular curve}.
The tangential direction of the singular curve is
called the {\it singular direction}, and 
the direction of the kernel of
$df$ is called the {\em  null direction}.
Let $\eta(t)$ be the smooth (non-vanishing)
vector field along the singular curve
which gives the null direction.

Here, we give examples:
A singular point is called a 
{\em cuspidal edge\/} or a {\em swallowtail\/} 
if the corresponding germ of $C^\infty$-map is 
right-left equivalent to that of $C^\infty$-map
\begin{equation}\tag{1}\label{eq:cuspidal-swallow}
 f_C(u,v):=(u^2,u^3,v) \quad\text{or}\quad
 f_S(u,v):=(3u^4+u^2v,4u^3+2uv,v)
\end{equation}
at $(u,v)=(0,0)$, respectively.
Here, two $C^{\infty}$-maps $f\colon{}(U,p)\to \R^3$ and 
$g\colon{}(V,q)\to \R^3$ are {\it right-left equivalent}
 at the
points $p\in U$ and $q\in V$ if there 
exists a local diffeomorphism $\varphi$ of $\R^2$ with $\varphi(p)=q$ 
and a local diffeomorphism $\Psi$ of $\R^3$ with 
$\Psi(f(p))=g(q)$
such that $g=\Psi\circ f \circ \varphi^{-1}$.
It can be easily checked that both of $f_C$ and $f_S$
are fronts whose singular sets are all non-degenerate. 
These two types of singular points characterize 
the generic singularities of wave fronts.
The singular curve of $f_C$ is the $v$-axis and
the null direction is the $u$-direction.
The singular curve of $f_S$ is the parabola $6u^2+v=0$ and
the null direction is the $u$-direction.
The following criteria are known:

\begin{Fact}[\cite{KRSUY}]\label{fact:krsuy}
Let $f:M^2\to \R^3$ be a front.
Let $p\in M^2$ be a non-degenerate singular point,
and $\gamma(t)$ the singular curve
of $f$ such that $\gamma(0)=p$. Then
\begin{enumerate}
\item 
$p$
is a cuspidal edge
if and only if  
the null direction $\eta(0)$ is transversal to
the singular direction $\dot\gamma(0)$,
\item 
$p$
is a swallowtail if and only if  
the null direction $\eta(0)$ is proportional to
the singular direction $\dot\gamma(0)$,
and satisfies
$$
\left.\frac{d}{dt}\right|_{t=0}
\det(\dot\gamma(t), \eta(t))\ne 0.
$$
\end{enumerate}
\end{Fact}

A {\it cuspidal cross cap} is a singular point 
which is right-left equivalent to the 
$C^\infty$-map 
\begin{equation}\label{eq:ccr2}
  f^{}_{CCR}(u,v):= (u,v^2,uv^3),
\end{equation}
which is not a front but a frontal
with unit normal vector field
\[
  \nu^{}_{CCR}:=\frac{1}{\sqrt{4+9u^2v^2+4v^6}}
                          (-2v^3,-3uv,2).
\]
Using (1) of Fact \ref{fact:krsuy},
one can easily check that all of singular points of
$f^{}_{CCR}$ except for the origin
consist of cuspidal edges.

\section*{Acknowledgments}
The authors are grateful to 
Wayne Rossman and the referee
for valuable comments.

\end{document}